\documentclass[10pt]{article}

\usepackage{fullpage}
\usepackage{amsmath}
\usepackage{amsthm}
\usepackage{amsfonts}  
\usepackage{amssymb}
\usepackage{bbm}
\usepackage{bm}
\usepackage{enumitem}
\usepackage{hyperref}
\usepackage{tikz}
\usepackage{pgfplots}
\usepackage{caption}
\usepackage{multirow}
\usepackage[nottoc]{tocbibind}

\usepackage[draft]{todonotes} 
\usepackage{comment}

\numberwithin{equation}{section}
\newtheoremstyle{dotless}{}{}{\itshape}{}{\bfseries}{}{ }{}
\theoremstyle{dotless}
\newtheorem{lem}{Lemma}[section]
\newtheorem*{lem*}{Lemma}
\newtheorem{prop}{Proposition}[section]
\newtheorem*{prop*}{Proposition}
\newtheorem{cor}{Corollary}[section]
\newtheorem*{cor*}{Corollary}
\newtheorem{thm}{Theorem}[section]
\newtheorem*{thm*}{Theorem}
\newtheorem{thmL}{Theorem}

\newtheorem{rem}{Remark}[section]
\newtheorem{definition}{Definition}[section]

\newcommand{\al}{\alpha}

\newcommand{\de}{\delta}
\newcommand{\De}{\Delta}
\newcommand{\ep}{\epsilon}
\newcommand{\la}{\lambda}
\newcommand{\La}{\Lambda}

\newcommand{\te}{\theta}
\newcommand{\Te}{\Theta}
\newcommand{\om}{\omega}
\newcommand{\ga}{\gamma}
\newcommand{\Ga}{\Gamma}
\newcommand{\dd}[1]{\hspace{.5mm}\mathrm{d}#1\hspace{.5mm}}

\newcommand{\R}{\mathbb{R}}

\newcommand{\ak}{\mathfrak{a}}
\newcommand{\ek}{\mathfrak{e}}
\newcommand{\bk}{\mathfrak{b}}
\newcommand{\mk}{\mathfrak{m}}
\newcommand{\uk}{\mathfrak{u}}
\newcommand{\vs}{\vspace{3mm}}
\newcommand{\no}{\noindent}
\newcommand{\rt}[1]{\sqrt{#1}\hspace{1mm}}

\title{Scaled penalization of Brownian motion\\ with drift and the Brownian ascent}
\author{Hugo Panzo\footnote{Department of Mathematics, University of Connecticut, \href{mailto:hugo.panzo@uconn.edu}{hugo.panzo@uconn.edu}}}

\begin{document}

\maketitle

\begin{abstract}
We study a scaled version of a two-parameter Brownian penalization model introduced by Roynette-Vallois-Yor in \cite{long_bridges}. The original model penalizes Brownian motion with drift $h\in\R$ by the weight process ${\big(\exp(\nu S_t):t\geq 0\big)}$ where $\nu\in\R$ and $\big(S_t:t\geq 0\big)$ is the running maximum of the Brownian motion. It was shown there that the resulting penalized process exhibits three distinct phases corresponding to different regions of the $(\nu,h)$-plane. In this paper, we investigate the effect of penalizing the Brownian motion concurrently with scaling and identify the limit process. This extends a result of \cite{Penalising} for the ${\nu<0,~h=0}$ case to the whole parameter plane and reveals two additional ``critical" phases occurring at the boundaries between the parameter regions. One of these novel phases is Brownian motion conditioned to end at its maximum, a process we call the \emph{Brownian ascent}. We then relate the Brownian ascent to some well-known Brownian path fragments and to a random scaling transformation of Brownian motion recently studied in \cite{pseudo_msj}.
\end{abstract}

\section{Introduction}

Brownian penalization was introduced by Roynette, Vallois, and Yor in a series of papers where they considered limit laws of the Wiener measure perturbed by various weight processes, see the monographs \cite{global, Penalising} for a complete list of the early works. One motivation for studying Brownian penalizations is that they can be seen as a way to condition Wiener measure by an event of probability $0$. Another reason penalizations are interesting is that they often exhibit phase transitions typical of statistical mechanics models. Let $\mathcal{C}\big(\R_+;\R\big)$ denote the space of continuous functions from $[0,\infty)$ to $\R$ and let $X=(X_t:t\geq 0)$ denote the canonical process on this space. For each $t\geq 0$, let ${\cal F}_t$ denote the $\sigma$-algebra generated by $(X_s:s\le t)$, and let ${\cal F}_\infty$ denote the $\sigma$-algebra generated by $\cup_{t\ge 0} {\cal F}_t$. We write ${\cal F}$ for the filtration $({\cal F}_t:t\ge0)$. Let $P_0$ denote the Wiener measure on $\mathcal{F}_\infty$, that is the unique measure under which $X$ is standard Brownian motion. As usual, we write $E_0$ for the corresponding expectation. Our starting point is the following definition of Brownian penalization:

\begin{definition}[\cite{Profeta}]
\label{def:penalized}
Suppose the $\mathcal{F}$-adapted weight process $\Gamma=(\Ga_t:t\geq 0)$ takes non-negative values and that $0<E_0[\Ga_t]<\infty$ for $t\geq 0$. For each $t\geq 0$, consider the Gibbs probability measure on $\mathcal{F}_\infty$ defined by
\begin{equation}\label{eq:gibbs}
Q_t^\Ga(\La):=\frac{E_0[1_{\La}\Ga_t]}{E_0[\Ga_t]},~\La\in\mathcal{F}_\infty.
\end{equation}
We say that the weight process $\Ga$ satisfies the \textbf{penalization principle for Brownian motion} if there exists a probability measure $Q^\Ga$ on $\mathcal{F}_\infty$ such that:
\begin{equation}\label{eq:limit}
\forall s\geq 0,~\forall\La_s\in\mathcal{F}_s,~~\lim_{t\to\infty}Q_t^\Ga(\La_s)=Q^\Ga(\La_s).
\end{equation}
In this case $Q^\Ga$ is called \textbf{Wiener measure penalized by} $\bm{\Ga}$ or \textbf{penalized Wiener measure} when there is no ambiguity. Similarly, $(X_t:t\geq 0)$ under $Q^\Ga$ is called \textbf{Brownian motion penalized by} $\bm{\Ga}$ or \textbf{penalized Brownian motion}.
\end{definition}

\begin{rem}
Let $P_x$ denote the distribution of Brownian motion starting at $x\in\R$ and $E_x$ its corresponding expectation. By replacing $E_0$ with $E_x$ in \eqref{eq:gibbs}, it is straightforward to modify the above definition of $Q_t^\Ga$ and $Q^\Ga$ to yield the analogous measures $Q_{x,t}^\Ga$ and $Q_x^\Ga$ which account for a general starting point. Since the main focus of our work is on computing scaling limits, and any fixed starting point would be scaled to $0$, we restrict our attention to penalizing $P_0$ for the sake of clarity.
\end{rem}

In many cases of interest, a one-parameter family of weight processes $(\Gamma^\nu:\nu \in {\mathbb R})$ is considered. The parameter $\nu$ allows us to adjust or ``tune" the strength of penalization and plays a role similar to that of the inverse temperature in statistical mechanics models. In this case we write $Q_t^\nu$ and $Q^\nu$ for the corresponding Gibbs and penalized measures, respectively. This notational practice is modified accordingly when dealing with a two-parameter weight process.

A natural choice for $\Ga$ are non-negative functions of the running maximum $S_t=\sup_{s\leq t}X_s$ and these \emph{supremum penalizations} were considered by the aforementioned authors in \cite{supremum}. Their results include as a special case the one-parameter weight process $\Ga_t=\exp(\nu S_t)$ with $\nu<0$ which we briefly describe here. 

\begin{thmL}[Roynette-Vallois-Yor \cite{supremum} Theorem 3.6]\label{thmL:supremum}\ \\
Let $\Ga_t=\exp(\nu S_t)$ with $\nu<0$. Then $\Ga$ satisfies the penalization principle for Brownian motion and the penalized Wiener measure $Q^\nu$ has the representation 
$$Q^\nu(\La_t)=E_0\left[1_{\La_t}M_t^\nu\right],~\forall t\geq 0,~\forall\La_t\in{\cal F}_t$$
where $M^\nu=(M^\nu_t:t\geq 0)$ is a positive $({\cal F},P_0)$-martingale starting at $1$ which has the form
\begin{equation}
M_t^\nu=\exp(\nu S_t)-\nu\exp(\nu S_t)(S_t-X_t).
\end{equation}
\end{thmL}

\begin{rem}
$M^\nu$ is an example of an \textbf{Az\'{e}ma-Yor martingale}. More generally, 
$$F(S_t)-f(S_t)(S_t-X_t),~t\geq 0$$
is an $({\cal F},P_0)$-local martingale whenever $f$ is locally integrable and $F(y)=\int_0^y f(x)\dd{x}$, see Theorem 3 of \cite{beyond_harmonic}.
\end{rem}

\no Let $S_\infty=\lim_{t\to\infty}S_t$ which exists in the extended sense due to monotonicity. Then the measure $Q^\nu$ can be further described in terms of a path decomposition at $S_\infty$.

\begin{thmL}[Roynette-Vallois-Yor \cite{supremum} Theorem 4.6]\label{thmL:path_decomp}\ \\
Let $\nu<0$ and $Q^\nu$ be as in \autoref{thmL:supremum}. 
\begin{enumerate}
\item Under $Q^\nu$, the random variable $S_\infty$ is exponential with parameter $-\nu$.
\item Let $g=\sup\{t\geq 0:X_t=S_\infty\}$. Then $Q^\nu(0<g<\infty)=1$, and under $Q^\nu$ the following hold:
			\begin{enumerate}
			\item the processes $(X_t:t\leq g)$ and $(X_g-X_{g+t}:t\geq 0)$ are independent,
			\item $(X_g-X_{g+t}:t\geq 0)$ is distributed as a Bessel(3) process starting at $0$,
			\item conditionally on $S_\infty =y>0$, the process $(X_t:t\leq g)$ is distributed as a Brownian motion started at $0$ and stopped when it first hits the level $y$.
			\end{enumerate}
\end{enumerate}
\end{thmL} 

More recently, penalization has been studied for processes other than Brownian motion. In this case, the measure of the underlying processes is referred to as the \emph{reference measure}. Similar supremum penalization results have been attained for simple random walk \cite{Debs1,Debs2}, stable processes \cite{stable_pen1,stable_pen2}, and integrated Brownian motion \cite{Profeta}. This paper builds upon the work of Roynette-Vallois-Yor that appeared in \cite{long_bridges} where they penalized Brownian motion using the two-parameter weight process ${\Ga_t=\exp(\nu S_t+hX_t)}$ with $\nu,h\in\R$. An easy application of Girsanov's theorem shows that this is equivalent to using the weight process $\Ga_t=\exp(\nu S_t)$ with $\nu\in\R$ and replacing the reference measure $P_0$ by the distribution of Brownian motion with drift $h$, henceforth referred to as $P_0^h$. Now their two-parameter penalization can also be seen as supremum penalization of Brownian motion with drift. In this model, the two parameters $\nu$ (for penalization) and $h$ (for drift) can have competing effects, leading to interesting phase transitions. That is to say the penalized Wiener measure is qualitatively different depending on where $(\nu,h)$ lies in the parameter plane. To describe the resulting phases, we first partition the parameter plane into six disjoint regions (the origin is excluded to avoid trivialities).

\no\begin{minipage}{0.5\textwidth}
\begin{itemize} 
\item $L_1=\{(\nu,h):\nu<0,~h=0\}$;
\item $R_1=\{(\nu,h):h<-\nu,~h>0\}$;
\item $L_2=\{(\nu,h):h=-\nu,~\nu<0\}$; 
\item $R_2=\{(\nu,h):h>-\nu,~h>-\frac{1}{2}\nu\}$;
\item $L_3=\{(\nu,h):h=-\frac{1}{2}\nu,~\nu>0\}$; and 
\item $R_3=\{(\nu,h):h<0,~h<-\frac{1}{2}\nu\}$. 
\end{itemize} 
\end{minipage}%
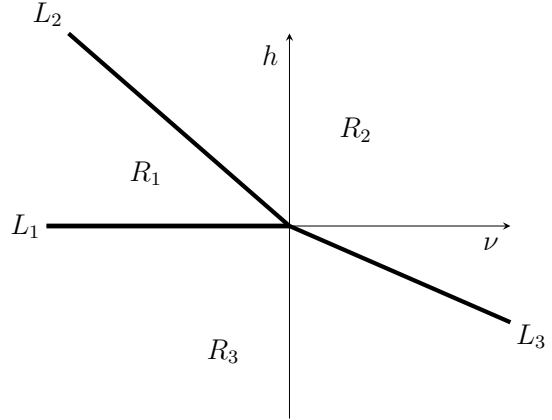
\begin{minipage}{0.5\textwidth}
\centering
\captionsetup{type=figure}
\captionof{figure}{phase diagram}
\label{fig:diagram}
\begin{tikzpicture}[scale=0.9]
\begin{axis}[axis lines = middle, clip=false, xtick=\empty, ytick=\empty, ymin=-2]
\addplot[sharp plot, line width=2pt] coordinates{(0,0) (-2.2,0)};
\addplot[sharp plot, line width=1.8pt] coordinates{(0,0) (-2,2)};
\addplot[sharp plot, line width=1.8pt] coordinates{(0,0) (2,-1)};
\node[label={[font=\large]340:{$\nu$}}] at (axis cs:1.6,0) {};
\node[label={[font=\large]340:{$h$}}] at (axis cs:-.4,2.0) {};
\node[label={[font=\large]180:{$R_1$}}] at (axis cs:-1,.55) {};
\node[label={[font=\large]180:{$R_3$}}] at (axis cs:-.3,-1.3) {};
\node[label={[font=\large]0:{$R_2$}}] at (axis cs:.3,1) {};
\node[label={[font=\large]340:{$L_3$}}] at (axis cs:1.9,-.9) {};
\node[label={[font=\large]135:{$L_2$}}] at (axis cs:-1.9,1.9) {};
\node[label={[font=\large]180:{$L_1$}}] at (axis cs:-2.1,0) {};
\end{axis}
\end{tikzpicture}
\end{minipage}

\begin{thmL}[Roynette-Vallois-Yor \cite{long_bridges} Theorem 1.7]\label{thmL:2parameters}\ \\
Let $\Ga_t=\exp(\nu S_t+hX_t)$ with $\nu,h\in\R$. For $\nu<0$, let $Q^\nu$ be as in \autoref{thmL:supremum}. Then $\Ga$ satisfies the penalization principle for Brownian motion and the penalized Wiener measure $Q^{\nu,h}$ has three phases which are given by
$$\dd{Q}^{\nu,h}=\left\{
  \begin{array}{ll}
    \dd{Q}^{\nu+h} & : (\nu,h)\in L_1\cup R_1\\
		\\
    \dd{P}_0^{\nu+h} & : (\nu,h)\in L_2\cup R_2\cup L_3\\ 
		\\
    \frac{\nu+2h}{2h}\exp(\nu S_\infty)\dd{P}_0^h & : (\nu,h)\in R_3.\\ 		
  \end{array}
  \right.$$
\end{thmL}

\begin{rem}
While \autoref{thmL:2parameters} had only three phases hence needing only three regions in the parameter plane, the rationale behind our choice of six regions in the phase diagram \autoref{fig:diagram} will become clear when we state our main results.
\end{rem}

One topic that has attracted interest is the following: to what extent can sets $\La\in\mathcal{F}_\infty$ replace the sets $\La_s\in\mathcal{F}_s$ in the limit \eqref{eq:limit} which defines penalization? This can't be done in complete generality, e.g. the case of supremum penalization from \autoref{thmL:path_decomp}. Here we saw that $S_\infty$ is exponentially distributed under $Q^\nu$, yet it is easy to see that for each $t\geq 0$, $S_\infty=\infty$ almost surely under $Q^\nu_t$. The last chapter of \cite{Penalising} is devoted to studying this question. Intuitively speaking, the effect of penalizing by $\Ga_t$ on the probability of an event $\La_s\in\mathcal{F}_s$ should be more pronounced the closer $s$ is to $t$. Hence letting $s$ keep pace with $t$ as $t\to\infty$ instead of having $s$ remain fixed might result in a different outcome. This leads to the notion of \emph{scaled Brownian penalizations}. Roughly speaking, scaled penalization amounts to penalizing and scaling simultaneously. To be more specific, we first introduce the family of scaled processes $X^{\alpha,t} = (X_s^{\al,t}:0\leq s\leq 1)$ indexed by $t>0$ with scaling exponent $\alpha\ge 0$ where $X_s^{\al,t}:=X_{st}/t^\al$. Now with a weight process $\Ga$ that satisfies the penalization principle for Brownian motion and the right choice of $\al$, we then compare the weak limits of $X^{\alpha,t}$ under $Q^\Ga|_{\mathcal{F}_t}$ (penalizing then scaling) and under $Q_t^\Ga|_{\mathcal{F}_t}$ (scaling and penalizing simultaneously) as $t\to\infty$. Next we give an example of such a result for supremum penalization.

\begin{thmL}[Roynette-Yor \cite{Penalising} Theorem 4.18]\label{thmL:scaled}\ \\
Let $(R_s:0\leq s\leq 1)$ and $(\mk_s:0\leq s\leq 1)$ denote a Bessel(3) process and Brownian meander, respectively. Let $\Ga_t=\exp(\nu S_t)$ with $\nu<0$ and $Q^\nu$ be the penalized Wiener measure. Then the distribution of $X^{\frac{1}{2},t}$ under $Q^\nu|_{\mathcal{F}_t}$ and $Q_t^\nu|_{\mathcal{F}_t}$ converges weakly to $(-R_s:0\leq s\leq 1)$ and $(-\mk_s:0\leq s\leq 1)$, respectively, as $t\to\infty$.
\end{thmL}

\no Notice how the scaling limits are similar (their path measures are mutually absolutely continuous by the Imhof relation \eqref{eq:Imhof}) yet at the same time quite different (the Bessel(3) process is a time-homogenous Markov process while the Brownian meander is a time-inhomogeneous Markov process). Since $Q^\nu$ is just a special case of $Q^{\nu,h}$, a natural question is whether results similar to \autoref{thmL:scaled} can be proven for the two-parameter model from \autoref{thmL:2parameters}. This is the primary goal of this paper.

\subsection{Main Results}
Our main results lie in computing the weak limits of $X^{\alpha,t}$ under $Q^{\nu,h}|_{\mathcal{F}_t}$ and $Q_t^{\nu,h}|_{\mathcal{F}_t}$ as $t\to\infty$ for all $(\nu,h)\in\R^2$ with appropriate $\al$. This extends the scaled penalization result of \autoref{thmL:scaled} to the whole parameter plane of the two-parameter model from \autoref{thmL:2parameters}. In doing so, we reveal two additional ``critical" phases. These new phases correspond to the parameter rays $L_2$ and $L_3$ which occur at the interfaces of the other regions,  see  \autoref{fig:diagram}. We call the two novel processes corresponding to these critical phases the \emph{Brownian ascent} and the \emph{up-down process}: 
\begin{itemize} 
\item Loosely speaking, the Brownian ascent $(\ak_s:0\leq s\leq 1)$ is a Brownian path of duration $1$ conditioned to end at its maximum, i.e. conditioned on the event $\{X_1=S_1\}$. It can be represented as a path transformation of the Brownian meander, see \autoref{sec:ascent}. While the sobriquet \emph{ascent} is introduced in the present paper, this transformed meander process first appeared in \cite{fp_bridge} in the context of Brownian \emph{first passage bridge}. More recently, the same transformed meander has appeared in \cite{pseudo_msj}, although no connection is made in that paper to \cite{fp_bridge} or Brownian motion conditioned to end at its maximum.

\item The up-down process is a random mixture of deterministic up-down paths that we now describe. For $0\leq \Te\leq 1$, let $\uk^\Te=(\uk_s^\Te:0\leq s\leq 1)$ be the continuous path defined by 
$$\uk_s^\Te=\Te-|s-\Te|,~0\leq s\leq 1.$$
It is easy to see that $\uk^\Te$ linearly interpolates between the points $(0,0)$, $(\Te,\Te)$, and $(1,2\Te-1)$ when $0<\Te<1$, while $\uk^0$ or $\uk^1$ is simply the path of constant slope $-1$ or $1$, respectively. To simplify notation, we write $\uk$ instead of $\uk^1$. The up-down process with slope $h$ is defined as the process $h\uk^U$ where $U$ is a Uniform$[0,1]$ random variable. Roughly speaking, the up-down process of slope $h$ starts at $0$ and moves with slope $h$ until a random Uniform$[0,1]$ time $U$ then it moves with slope $-h$ for the remaining time $1-U$. In our results the initial slope is always positive, hence the name up-down. 
\end{itemize} 

\begin{thm}\label{thm:theorem_1}
Let $W=(W_s:0\leq s\leq 1)$ and $R=(R_s:0\leq s\leq 1)$ be Brownian motion and Bessel(3) processes, each starting at $0$. Let $\ak$ be a Brownian ascent, $\mk$ be a Brownian meander, and $\ek$ be a normalized Brownian excursion. Let $\uk^U$ be the up-down process with $U$ a Uniform$[0,1]$ random variable. Then the weak limits of $X^{\alpha,t}$ as $t\to\infty$ are as follows:
\begin{center}
\resizebox{\textwidth}{!}{
\begin{tabular}{|c|c|c|c|c|}\hline
\textbf{Region} & $\bm{\al}$ & \textbf{Limit under} $\bm{Q^{\nu,h}|_{\mathcal{F}_t}}$ & \textbf{Limit under} $\bm{Q_t^{\nu,h}|_{\mathcal{F}_t}}$ & \textbf{Proof} \\
\hline
$L_1$ & \multirow{3}{*}{$1/2$} & \multirow{2}{*}{$-R$}  & $-\mk$ & \autoref{thmL:scaled} \\
\cline{1-1}\cline{4-5}
$R_1$ &  &  & $-\ek$ & \autoref{sec:R_1} \\
\cline{1-1}\cline{3-5}
$L_2$ &  & $W$  & $\ak$ & \autoref{sec:L_2} \\
\hline
$R_2$  & \multirow{3}{*}{$1$} & \multicolumn{2}{c|}{$(\nu+h)\uk$} & \autoref{sec:R_2,3} \\ 
\cline{1-1}\cline{4-5}
$L_3$   &    			      &               &  $-h\uk^U$ & \autoref{sec:L_3} \\
\cline{1-1}\cline{3-5}
$R_3$ & &  \multicolumn{2}{c|}{$h\uk$ } & \autoref{sec:R_2,3} \\
\hline
\end{tabular}}
\end{center}
\end{thm}

\no Some remarks are in order: 
\begin{rem} The $L_1$ row is  a restatement of \autoref{thmL:scaled}. All remaining rows are new. \end{rem}
\begin{rem} Note the abrupt change in behavior when going from $L_3$ to $R_3$ in the $Q^{\nu,h}|_{\mathcal{F}_t}$ column, namely going from slope $\nu+h=-h$ to slope $h$. However, in the $Q_t^{\nu,h}|_{\mathcal{F}_t}$ column, i.e. the scaled penalization, the change from $R_2$ to $L_3$ to $R_3$ is in some sense less abrupt because the up-down process switches from $R_2$ behavior (slope $\nu+h=-h$) to $R_3$ behavior (slope $h$) at a random Uniform$[0,1]$ time.
\end{rem}

With the exception of the up-down process, the limits in the $\al=1$ regions (ballistic scaling) are all deterministic. This suggests that we try subtracting the deterministic drift from the canonical process and then scale this centered process. This leads to a functional central limit theorem which is proved in \autoref{sec:theorem_2}.

\begin{thm}\label{thm:theorem_2}
Let $W=(W_s:0\leq s\leq 1)$ be Brownian motion starting at $0$. Then the weak limits of the centered and scaled canonical process as $t\to\infty$ are as follows:
\begin{center}
\resizebox{\textwidth}{!}{
\begin{tabular}{|c|c|c|c|}\hline
\textbf{Region} & \textbf{Process} & \textbf{Limit under} $\bm{Q^{\nu,h}|_{\mathcal{F}_t}}$ & \textbf{Limit under} $\bm{Q_t^{\nu,h}|_{\mathcal{F}_t}}$ \\ \hline
$R_2$ & \multirow{2}{*}{$\frac{X_{\bullet t}-(\nu+h) t\uk_\bullet}{\rt{t}}$} &  \multicolumn{2}{c|}{}  \\
\cline{1-1}\cline{4-4}
$L_3$ &  & W & see \autoref{thm:theorem_3}  \\
\cline{1-2}\cline{4-4}
$R_3$ & $\frac{X_{\bullet t}-ht\uk_\bullet}{\rt{t}}$   & \multicolumn{2}{c|}{}  \\
\hline
\end{tabular}}
\end{center}
\end{thm} 
\no Here the $\bullet$ stands for a time parameter that ranges between $0$ and $1$. For example, this notation allows us to write $X_{\bullet t}-ht\uk_\bullet$ rather than the unwieldy $(X_{st}-ht\uk_s:0\leq s\leq 1)$.

The center-right entry in the above table is exceptional because in that case we don't have a deterministic drift that can be subtracted in order to center the process. In light of the $L_3$ row of \autoref{thm:theorem_1}, we need to subtract a random drift that switches from $-h$ to $h$ at an appropriate time. More specifically, define ${\Te_t=\inf\{s:X_s=S_t\}}$. Then for each $t>0$, we want to subtract from $X_{\bullet t}$ a continuous process starting at $0$ that has drift $-ht$ until time $\Te_t/t$ and has drift $ht$ thereafter. A candidate for this random centering process is $2S_{\bullet t}+ht\uk_\bullet$ which is seen to have the desired up-down drift. Indeed, this allows us to fill in the remaining entry of the table with a functional central limit theorem which is proved in \autoref{sec:theorem_3}.

\begin{thm}\label{thm:theorem_3}
Let $W=(W_s:0\leq s\leq 1)$ be Brownian motion starting at $0$. If $(\nu,h)\in L_3$ then 
$$\frac{X_{\bullet t}-(2S_{\bullet t}+ht\uk_\bullet)}{\rt{t}}$$
under $Q_t^{\nu,h}|_{\mathcal{F}_t}$ converges weakly to $W$ as $t\to\infty$.
\end{thm}

\subsection{Notation}\label{sec:notation}
For the remainder of this paper, the non-negative real numbers are denoted by $\R_+$ and we use $X$ for the $\mathcal{C}\big(\R_+;\R\big)$ canonical process under $P_x$ or $P_x^h$. Abusing notation, we also use $X$ for the $\mathcal{C}\big([0,1];\R\big)$ canonical process under the analogous measures $P_x$ or $P_x^h$. We use a subscript when restricting to paths starting at a particular point, say $\mathcal{C}_0\big([0,1];\R\big)$ for paths starting at $0$. As explained below \autoref{thm:theorem_2}, we often use the bullet point $\bullet$ to stand for a time parameter that ranges between $0$ and $1$. Expectation under $P_x$ and $P_x^h$ is denoted by $E_x$ and $E_x^h$, respectively. Expectation under $Q^\nu$ and $Q_t^\nu$ is denoted by $Q^\nu[\cdot]$ and $Q_t^\nu[\cdot]$, respectively. We always have $S_t$ and $I_t$ denoting the running extrema of the canonical process, namely $S_t=\sup_{s\leq t}X_s$ and $I_t=\inf_{s\leq t}X_s$. When dealing with the canonical process, we use $\Te_t$ and $\te_t$ to refer to the first time at which the maximum and minimum, respectively, is attained over the time interval $[0,t]$. Also, $\tau_x$ and $\ga_x$ refer to the first and last hitting times of $x$, respectively. More specifically, $\Te_t=\inf\{s:X_s=S_t\}$, $\te_t=\inf\{s:X_s=I_t\}$, $\tau_x=\inf\{t:X_t=x\}$, and $\ga_x=\sup\{t:X_t=x\}$. For the Bessel(3) process we use $R$. A Brownian path of duration $1$ starting at $0$ is written as $W$. We use Fraktur letters for the various time-inhomogeneous processes that appear in this paper: Brownian ascent $\ak$, standard Brownian bridge $\bk$, pseudo-Brownian bridge $\tilde{\bk}$, normalized Brownian excursion $\ek$, Brownian meander $\mk$ and co-meander $\tilde{\mk}$, up-down process $\uk^U$, and also for the deterministic path with slope $1$ which we write as $\uk$. The path transformations $\phi$ and $\Pi_\de$ are described in \autoref{def:reversal} and \autoref{def:F_delta}, respectively. We use $\|F\|$ to denote the supremum norm of a bounded path functional $F:\mathcal{C}\big([0,1];\R\big)\to\R$.

\subsection{Organization of Paper}
In \autoref{sec:duality}, we introduce a duality relation between regions of the parameter plane and use it to compute the partition function asymptotics. \autoref{sec:theorem_1} is devoted to the proof of \autoref{thm:theorem_1}. In \autoref{sec:theorem_2}, we describe a modification of path functionals that allows us to decouple them from an otherwise dependent factor. This is used to prove \autoref{thm:theorem_2} in the remainder of that section. \autoref{sec:theorem_3} contains the proof of \autoref{thm:theorem_3}. In \autoref{sec:ascent}, we provide several constructions of the Brownian ascent and give a connection to some recent literature. In \autoref{sec:future}, we comment on work in preparation and suggest some directions for future research. We gather various known results used throughout the paper and include them along with references in \autoref{app:appendix}.

\section{Duality and Partition Function Asymptotics}\label{sec:duality}
The normalization constant $E_0[\Ga_t]$ appearing in the Gibbs measure \eqref{eq:gibbs} is known as the \emph{partition function}. The first step in proving a scaled penalization result is to obtain an asymptotic for the partition function as $t\to\infty$. Indeed, differing asymptotics may indicate different phases and the asymptotic often suggests the right scaling exponent, e.g. diffusive scaling for power law behavior and ballistic scaling for exponential behavior. In this section we compute partition function asymptotics for our two-parameter model in each of the parameter phases. While some of these asymptotics appeared in \cite{long_bridges}, we derive them all for the sake of completeness. Most of the cases boil down to an application of Watson's lemma or Laplace's method. Refer to \autoref{app:appendix} for precise statements of these tools from asymptotic analysis. 

To reduce the number of computations required, we make use of a duality relation between the regions $R_2$ and $R_3$ and critical lines $L_1$ and $L_2$. Recall that $P_0$ is invariant under the path transformation ${X_\bullet\mapsto X_{1-\bullet}-X_1}$, for example, see Lemma 2.9.4 in \cite{K&S}. So the joint distribution of $X_1$ and $S_1$ under $P_0$ coincides with that of $-X_1$ and $\sup_{s\leq 1}(X_{1-s}-X_1)=S_1-X_1$ under $P_0$. Consequently,
\begin{equation*}
\begin{split}
E_0\left[\exp\left(\nu S_t+hX_t\right)\right]&=E_0\left[\exp\left(\rt{t}(\nu S_1+hX_1)\right)\right]\\
&=E_0\left[\exp\left(\rt{t}\big(\nu S_1-(\nu+h)X_1\big)\right)\right]\\
&=E_0\left[\exp\left(\nu S_t-(\nu+h)X_t\right)\right]
\end{split}
\end{equation*}
follows from Brownian scaling. An easy calculation while referring to \autoref{fig:diagram} shows that if $(\nu,h)\in L_1$, then $\big(\nu,-(\nu+h)\big)\in L_2$ and vice versa. The same holds for $R_2$ and $R_3$. In fact, the map $(\nu,h)\mapsto\big(\nu,-(\nu+h)\big)$ is an involution from $R_2$ onto $R_3$ and from $L_1$ onto $L_2$. Hence the partition function asymptotics for $L_2$ and $R_2$ can be obtained from those of $L_1$ and $R_3$ by substituting $\big(\nu,-(\nu+h)\big)$ for $(\nu,h)$. 

We can push this idea further by applying it to functionals of the entire path. First we need some new definitions.

\begin{definition}\label{def:reversal}
For $X_\bullet\in\mathcal{C}\big([0,1];\R\big)$, define $\phi:\mathcal{C}\big([0,1];\R\big)\to \mathcal{C}\big([0,1];\R\big)$ by
$$\phi X_s=X_{1-s}-(X_1-X_0),~0\leq s\leq 1.$$
For $F:\mathcal{C}\big([0,1];\R\big)\to\R$, let $F_\phi$ denote $F\circ\phi$.
\end{definition}
This path transformation reverses time and shifts the resulting path so that it starts at the same place. It is clear that $\phi$ is a linear transformation from $\mathcal{C}\big([0,1];\R\big)$ onto $\mathcal{C}\big([0,1];\R\big)$ that is continuous and an involution. Additionally, $F_\phi$ is bounded continuous whenever $F$ is and $\|F_\phi\|=\|F\|$. As mentioned above, $P_0$ is invariant under $\phi$ so we have $E_0[F_\phi(X_\bullet)]=E_0[F(X_\bullet)]$. This duality will be exploited again in \autoref{sec:L_2}, \autoref{sec:R_2,3}, and \autoref{sec:theorem_2}.

\begin{prop}\label{prop:asymptotics}
The partition function has the following asymptotics as $t\to\infty$:
$$E_0[\exp(\nu S_t+hX_t)]\sim\left\{
  \begin{array}{ll}
    -\frac{1}{\nu}\rt{\frac{2}{\pi t}} & : L_1=\{(\nu,h):\nu<0,~h=0\}\\
    \\
    -\frac{\nu}{h^2(\nu+h)^2}\rt{\frac{2}{\pi t^3}} & : R_1=\{(\nu,h):h<-\nu,~h>0\}\\
    \\
    -\frac{1}{\nu}\rt{\frac{2}{\pi t}} & : L_2=\{(\nu,h):h=-\nu,~\nu<0\}\\
		\\
		2\frac{\nu+h}{\nu+2h}\exp\left(\frac{1}{2}(\nu+h)^2 t\right) & : R_2=\{(\nu,h):h>-\nu,~h>-\frac{1}{2}\nu\}\\
  	\\
  	2h^2t\exp\left(\frac{1}{2}h^2 t\right) & : L_3=\{(\nu,h):h=-\frac{1}{2}\nu,~\nu>0\}\\
    \\
    \frac{2h}{\nu+2h}\exp\left(\frac{1}{2}h^2 t\right) & : R_3=\{(\nu,h):h<0,~h<-\frac{1}{2}\nu\}.\\    
  \end{array}
  \right.$$
\end{prop}

\begin{proof}
We divide the proof into four cases.\vs
\begin{enumerate}[label=\textbf{\arabic*.}]
\item
\no\textbf{\textit{$\boldsymbol{L_1}$ and $\boldsymbol{L_2}$ case}}
\vs

The $L_1$ asymptotic follows from Watson's lemma being applied to 
$$E_0\left[\exp\left(\nu\rt{t}S_1\right)\right]=\int_0^\infty\exp\left(\nu\rt{t}y\right)\frac{2}{\rt{2\pi}}\exp\left(-\frac{y^2}{2}\right)\dd{y}$$
where $\nu<0$. For $(\nu,h)\in L_2$, we appeal to duality and substitute $\big(\nu,-(\nu+h)\big)=(\nu,0)$ for $(\nu,h)$ in the $L_1$ asymptotic.

\item
\no\textbf{\textit{$\boldsymbol{R_2}$ and $\boldsymbol{R_3}$ case}}
\vs

Let $(\nu,h)\in R_3$. Using a Girsanov change of measure, we can write
$$E_0\left[\exp\left(\nu S_t+hX_t\right)\right]=E_0^h\left[\exp\left(\nu S_t\right)\right]\exp\left(\frac{1}{2}h^2 t\right).$$
In $R_3$ we have $h<0$ so $S_\infty$ is almost surely finite under $P_0^h$. In fact, $S_\infty$ has the Exponential$(-2h)$ distribution under $P_0^h$. This follows immediately from Williams' path decomposition \autoref{thm:Williams}. Since $\nu<-2h$, dominated convergence implies
$$\lim_{t\to\infty}E_0^h\left[\exp\left(\nu S_t\right)\right]=E_0^h\left[\exp\left(\nu S_\infty\right)\right]=\frac{2h}{\nu+2h}.$$
This gives the $R_3$ asymptotic. Once again, we can use duality to get the $R_2$ asymptotic by substituting $\big(\nu,-(\nu+h)\big)$ for $(\nu,h)$ in the $R_3$ asymptotic.

\item\label{itm:L_3}
\no\textbf{\textit{$\boldsymbol{L_3}$ case}}
\vs

In $L_3$ we have $(\nu,h)=(-2h,h)$, hence Pitman's $2S-X$ theorem implies
$$E_0\left[\exp\left(\rt{t}(\nu S_1+hX_1)\right)\right]=E_0\left[\exp\left(\rt{t}(-2hS_1+hX_1)\right)\right]=E_0\left[\exp\left(-h\rt{t}R_1\right)\right]$$
where $(R_t:t\geq 0)$ is a Bessel(3) process. So using the Bessel(3) transition density \eqref{eq:Bes0_dens}, we can write the $L_3$ partition function as
\begin{equation}\label{eq:L3_part}
\int_0^\infty\exp\left(-h\rt{t}y\right)\rt{\frac{2}{\pi}}y^2\exp\left(-\frac{y^2}{2}\right)\dd{y}.
\end{equation}
After the change of variables $y\mapsto y\rt{t}$, we arrive at
$$\rt{\frac{2}{\pi}}t^\frac{3}{2}\int_0^\infty y^2\exp\left(-\Big(hy+\frac{y^2}{2}\Big)t\right)\dd{y}$$
whose asymptotic can be ascertained by a direct application of Laplace's method. 

\item
\no\textbf{\textit{$\boldsymbol{R_1}$ case}}
\vs

Let $(R_t:t\geq 0)$ denote a Bessel(3) process starting at $0$ and $U$ an independent Uniform$[0,1]$ random variable. For $t$ fixed, the identity in law
$$(S_t,S_t-X_t)\stackrel{\mathcal{L}}{=}(UR_t,(1-U)R_t)$$
follows from Pitman's $2S-X$ theorem, see Item C in Chapter 1 of \cite{Penalising}. When $\nu+2h\ne 0$, this identity leads to 
\begin{equation*}
\begin{split}
&E_0\left[\exp\left(\nu S_t+hX_t\right)\right]=E_0\left[\exp\big((\nu+2h)R_tU-hR_t\big)\right]\\
&=E_0\left[\exp\left(-h\rt{t}R_1\right)\int_0^1\exp\left((\nu+2h)\rt{t}R_1 u\right)\dd{u}\right]\\
&=\frac{1}{(\nu+2h)\rt{t}}\left(E_0\left[\frac{\exp\left((\nu+h)\rt{t}R_1\right)}{R_1}\right]-E_0\left[\frac{\exp\left(-h\rt{t}R_1\right)}{R_1}\right]\right)\\
&=\frac{1}{(\nu+2h)}\rt{\frac{2}{\pi t}}\left(\int_0^\infty\exp\left((\nu+h)\rt{t}y\right)y\exp\left(-\frac{y^2}{2}\right)\dd{y}-\int_0^\infty\exp\left(-h\rt{t}y\right)y\exp\left(-\frac{y^2}{2}\right)\dd{y}\right).
\end{split}
\end{equation*}
Since $h>0$ and $\nu+h<0$, Watson's lemma can be applied to both integrals and their asymptotics combined. If $\nu+2h=0$ instead, we can use the reasoning from case \ref{itm:L_3} to show that the partition function is equal to \eqref{eq:L3_part}. However, unlike that case, now we have $h>0$ so Watson's lemma can be applied to yield the desired asymptotic.
\end{enumerate}
\end{proof}

\section[Proof of Theorem 1.1]{Proof of \autoref{thm:theorem_1}}\label{sec:theorem_1}

\subsection[L2 case]{$L_2$ case}\label{sec:L_2}

In this section we prove the $L_2$ row in the table from \autoref{thm:theorem_1}. That the limit under $Q^{\nu,h}|_{\mathcal{F}_t}$ is $W$ follows trivially from Brownian scaling and \autoref{thmL:2parameters} since $Q^{\nu,h}=P_0^{\nu+h}=P_0$ when $(\nu,h)\in L_2$. To prove the limit under $Q_t^{\nu,h}|_{\mathcal{F}_t}$ is $\ak$, we show that 
\begin{equation}\label{eq:L_2}
\lim_{t\to\infty}Q_t^{\nu,h}\left[F\left(\frac{X_{\bullet t}}{\rt{t}}\right)\right]=E\left[F(\mk_1-\mk_{1-\bullet})\right]
\end{equation}
for any bounded continuous $F:\mathcal{C}\big([0,1];\R\big)\to\R$ and $(\nu,h)\in L_2$. Then the desired result follows from \autoref{prop:ascent}. 

\begin{proof}[Proof of \eqref{eq:L_2}]
The idea behind the proof is to use duality to transfer the result of \autoref{thmL:scaled} from the $L_1$ phase to the $L_2$ phase. Recall that in $L_2$ we have $\nu<0$ and $\nu+h=0$. Then the invariance property of $\phi$ and \autoref{thmL:scaled} imply that 
\begin{equation*}
\begin{split}
\lim_{t\to\infty}Q_t^{\nu,h}\left[F\left(\frac{X_{\bullet t}}{\rt{t}}\right)\right]&=\lim_{t\to\infty}\frac{E_0\left[F_\phi\left(\frac{X_{\bullet t}}{\rt{t}}\right)\exp\left(\nu S_t-(\nu+h)X_t\right)\right]}{E_0\left[\exp\left(\nu S_t-(\nu+h)X_t\right)\right]}\\
&=\lim_{t\to\infty}\frac{E_0\left[F_\phi\left(\frac{X_{\bullet t}}{\rt{t}}\right)\exp\left(\nu S_t\right)\right]}{E_0\left[\exp\left(\nu S_t\right)\right]}\\
&=E\left[F_\phi(-\mk_\bullet)\right]\\
&=E\left[F(\mk_1-\mk_{1-\bullet})\right].
\end{split}
\end{equation*}
\end{proof}

\subsection[R2 and R3 case]{$R_2$ and $R_3$ case}\label{sec:R_2,3}

In this section we prove the $R_2$ and $R_3$ rows in the table from \autoref{thm:theorem_1}. This is done by showing that the limits
\begin{equation}\label{eq:R2,3_Q}
\lim_{t\to\infty}Q^{\nu,h}\left[F\left(\frac{X_{\bullet t}}{t}\right)\right]=\left\{
  \begin{array}{ll}
    F\big((\nu+h)\uk_\bullet\big) & :(\nu,h)\in R_2\\
		\\
		F\big(h\uk_\bullet\big) & :(\nu,h)\in R_3\\
  \end{array}
  \right.
\end{equation} 
and
\begin{equation}\label{eq:R2,3_scaled}
\lim_{t\to\infty}Q_t^{\nu,h}\left[F\left(\frac{X_{\bullet t}}{t}\right)\right]=\left\{
  \begin{array}{ll}
    F\big((\nu+h)\uk_\bullet\big) & :(\nu,h)\in R_2\\
		\\
		F\big(h\uk_\bullet\big) & :(\nu,h)\in R_3\\
  \end{array}
  \right.
\end{equation}
hold for any bounded continuous $F:\mathcal{C}\big([0,1];\R\big)\to\R$. First we need a lemma which asserts that ballistic scaling of Brownian motion with drift $h$ results in a deterministic path of slope $h$.

\begin{lem}\label{lem:ballistic}
Let $h\in\R$. Then $X_{\bullet t}/t$ converges to $h\uk_\bullet$ in probability under $P_0^h$ as $t\to\infty$.
\end{lem}

\begin{proof}
For any $\ep>0$ we have
\begin{equation*}
\begin{split}
\lim_{t\to\infty}P_0^h\left(\left\|\frac{X_{\bullet t}}{t}-h\uk_\bullet\right\|>\ep\right)&=\lim_{t\to\infty}P_0\left(\left\|\frac{X_{\bullet t}+ht\uk_\bullet}{t}-h\uk_\bullet\right\|>\ep\right)\\
&=\lim_{t\to\infty}P_0\left(\left\|X_\bullet\right\|>\ep\rt{t}\right)\\
&=0
\end{split}
\end{equation*}
since $\|X_\bullet\|$ is almost surely finite under $P_0$.
\end{proof}
\vs

\begin{proof}[Proof of \eqref{eq:R2,3_Q}]
When $(\nu,h)\in R_2$, we can use \autoref{thmL:2parameters}, \autoref{lem:ballistic}, and bounded convergence to get 
\begin{equation*}
\begin{split}
\lim_{t\to\infty}Q^{\nu,h}\left[F\left(\frac{X_{\bullet t}}{t}\right)\right]&=\lim_{t\to\infty}E_0^{\nu+h}\left[F\left(\frac{X_{\bullet t}}{t}\right)\right]\\
&=F\left((\nu+h)\uk_\bullet\right).
\end{split}
\end{equation*}
When $(\nu,h)\in R_3$, we can use \autoref{thmL:2parameters}, \autoref{lem:ballistic}, and dominated convergence to get
\begin{equation*}
\begin{split}
\lim_{t\to\infty}Q^{\nu,h}\left[F\left(\frac{X_{\bullet t}}{t}\right)\right]&=\frac{\nu+2h}{2h}\lim_{t\to\infty}E_0^h\left[F\left(\frac{X_{\bullet t}}{t}\right)\exp(\nu S_\infty)\right]\\
&=\frac{\nu+2h}{2h}F\left(h\uk_\bullet\right)E_0^h\left[\exp(\nu S_\infty)\right]\\
&=F\left(h\uk_\bullet\right).
\end{split}
\end{equation*}
Here we used the fact that $S_\infty$ has the Exponential$(-2h)$ distribution under $P_0^h$ which follows from Williams' path decomposition.
\end{proof}
\vs

\begin{proof}[Proof of \eqref{eq:R2,3_scaled}]
We first show that the limit holds in the $R_3$ case and then use duality to transfer this result to the $R_2$ case. Accordingly, suppose $(\nu,h)\in R_3$. By a Girsanov change of measure, we can write
$$E_0\left[F\left(\frac{X_{\bullet t}}{t}\right)\exp\left(\nu S_t+hX_t\right)\right]=\exp\left(\frac{1}{2}h^2 t\right)E_0^h\left[F\left(\frac{X_{\bullet t}}{t}\right)\exp\left(\nu S_t\right)\right].$$
Dividing this by the partition function asymptotic from \autoref{prop:asymptotics} and using \autoref{lem:ballistic} with dominated convergence gives us
\begin{equation*}
\begin{split}
\lim_{t\to\infty}Q_t^{\nu,h}\left[F\left(\frac{X_{\bullet t}}{t}\right)\right]&=\frac{\nu+2h}{2h}\lim_{t\to\infty}E_0^h\left[F\left(\frac{X_{\bullet t}}{t}\right)\exp\left(\nu S_t\right)\right]\\
&=\frac{\nu+2h}{2h}F\big(h\uk_\bullet\big)E_0^h\left[\exp(\nu S_\infty)\right]\\
&=F\big(h\uk_\bullet\big).
\end{split}
\end{equation*}

Now suppose $(\nu,h)\in R_2$. Then $\big(\nu,-(\nu+h)\big)\in R_3$. Hence the invariance property of $\phi$ and the above result imply
\begin{equation*}
\begin{split}
\lim_{t\to\infty}Q_t^{\nu,h}\left[F\left(\frac{X_{\bullet t}}{t}\right)\right]&=\lim_{t\to\infty}Q_t^{\nu,-(\nu+h)}\left[F_\phi\left(\frac{X_{\bullet t}}{t}\right)\right]\\
&=F_\phi\big(-(\nu+h)\uk_\bullet\big)\\
&=F\big((\nu+h)\uk_\bullet\big).
\end{split}
\end{equation*}
\end{proof}

\subsection[L3 case]{$L_3$ case}\label{sec:L_3}

In this section we prove the $L_3$ row in the table from \autoref{thm:theorem_1}. The proof that the limit under $Q^{\nu,h}|_{\mathcal{F}_t}$ is $(\nu+h)\uk$ is identical to that of the $R_2$ case of \eqref{eq:R2,3_Q} since $Q^{\nu,h}=P_0^{\nu+h}$ when $(\nu,h)\in L_3$ by \autoref{thmL:2parameters}. To prove the limit under $Q_t^{\nu,h}|_{\mathcal{F}_t}$ is $-h\uk^U$, we show that
\begin{equation}\label{eq:L3_scaled}
\lim_{t\to\infty}Q_t^{\nu,h}\left[F\left(\frac{X_{\bullet t}}{t}\right)\right]=E\left[F\left(-h\uk_\bullet^U\right)\right]
\end{equation}
for any bounded Lipschitz continuous $F:\mathcal{C}\big([0,1];\R\big)\to\R$ and $(\nu,h)\in L_3$. 

First we need some preliminary results on Bessel(3) bridges and related path decompositions. Let $x,y\geq 0$ and $u>0$. The Bessel(3) bridge of length $u$ from $x$ to $y$ can be represented by
\begin{equation}\label{eq:bes_bridge}
\rt{\left(x+(y-x)\frac{s}{u}+\bk_s^{(1)}\right)^2+\left(\bk_s^{(2)}\right)^2+\left(\bk_s^{(3)}\right)^2},~0\leq s\leq u
\end{equation}
where $\bk^{(i)}$, $i=1,2,3$ are independent Brownian bridges of length $u$ from $0$ to $0$, see \cite{Pit_lec}. If $x,y>0$ and $0<u<1$, then the path $(X_s:0\leq s\leq 1)$ under $P_x$ conditionally given $\{(I_1,X_1,\te_1)=(0,y,u)\}$ can be decomposed into a concatenation of two Bessel(3) bridges. More precisely, the path fragments 
$$(X_s:0\leq s\leq u)\text{ and }(X_s:u\leq s\leq 1)$$
are independent and distributed respectively like
$$(R_s:0\leq s\leq u)\text{ given }\{(R_0,R_u)=(x,0)\}$$
and
$$(R_{s-u}:u\leq s\leq 1)\text{ given }\{(R_0,R_{1-u})=(0,y)\}.$$
This follows from Theorem 2.1.(ii) in \cite{max_decomp} and the discussion in the Introduction of \cite{fp_bridge}. 

\begin{lem}\label{lem:uniform}
Let $x,y>0$ and $0<u<1$. Consider the path $\om_{x,y}^u$ in $\mathcal{C}_0\big([0,1];\R\big)$ that linearly interpolates between the points $(0,0)$, $(u,x)$ and $(1,x-y)$. Specifically, $\om_{x,y}^u$ is given by 
$$\om_{x,y}^u(s) = \left\{
  \begin{array}{ll}
    x\frac{s}{u} & : 0\leq s\leq u\\
    \\
    x-y\frac{s-u}{1-u} & : u<s\leq 1.
  \end{array}
\right.
$$
Suppose $f(t)>0$ for $t>0$ and $\displaystyle\lim_{t\to\infty}f(t)=\infty$. If $F:\mathcal{C}\big([0,1];\R\big)\to\R$ is bounded Lipschitz continuous, then  
$$\lim_{t\to\infty}E_0\left[F\left(\frac{X_\bullet}{f(t)}\right)\middle|S_1=xf(t),S_1-X_1=yf(t),\Te_1=u\right]=F\left(\om_{x,y}^u(\bullet)\right)$$
and the convergence is uniform on $\{(x,y,u)\in \R^3:x>0,y>0,0<u<1\}$.
\end{lem}\vs

\begin{proof}
Using the translation and reflection symmetries of Wiener measure, we can write 
$$E_0\left[F\left(\frac{X_\bullet}{f(t)}\right)\middle|S_1=xf(t),S_1-X_1=yf(t),\Te_1=u\right]=E_{xf(t)}\left[F\left(x-\frac{X_\bullet}{f(t)}\right)\middle|I_1=0,X_1=yf(t),\te_1=u\right].$$
Together with \eqref{eq:bes_bridge} and the path decomposition noted above, this implies
\begin{equation}\label{eq:L3_bridges}
E_0\left[F\left(\frac{X_\bullet}{f(t)}\right)\middle|S_1=xf(t),S_1-X_1=yf(t),\Te_1=u\right]=E\left[F\left(x-\frac{Y^{(t)}_\bullet}{f(t)}\right)\right]
\end{equation}
where $Y^{(t)}$ is defined by
$$Y_s^{(t)}: = \left\{
  \begin{array}{ll}
   \rt{\left(xf(t)\frac{u-s}{u}+\bk_s^{(1)}\right)^2+\left(\bk_s^{(2)}\right)^2+\left(\bk_s^{(3)}\right)^2} & : 0\leq s\leq u\\
    \\
   \rt{\left(yf(t)\frac{s-u}{1-u}+\bk_{s-u}^{(4)}\right)^2+\left(\bk_{s-u}^{(5)}\right)^2+\left(\bk_{s-u}^{(6)}\right)^2} & : u<s\leq 1.
  \end{array}
\right.
$$
Here $\bk^{(i)}$, $1\leq i\leq 6$ are independent Brownian bridges from $0$ to $0$ of length $u$ or $1-u$ as applicable. Let $\|\cdot\|_2$ denote the Euclidean norm on $\R^3$. Then we have
$$\left|\om_{x,y}^u(s)-\left(x-\frac{Y_s^{(t)}}{f(t)}\right)\right|=\left\{
  \begin{array}{ll}
    \Bigg|\frac{1}{f(t)}\left\|\left(xf(t)\frac{u-s}{u}+\bk_s^{(1)},\bk_s^{(2)},\bk_s^{(3)}\right)\right\|_2-\left\|\Big(x\frac{u-s}{u},0,0\Big)\right\|_2\Bigg| & : 0\leq s\leq u\\
    \\
    \Bigg|\frac{1}{f(t)}\left\|\left(yf(t)\frac{s-u}{1-u}+\bk_{s-u}^{(4)},\bk_{s-u}^{(5)},\bk_{s-u}^{(6)}\right)\right\|_2-\left\|\left(y\frac{s-u}{1-u},0,0\right)\right\|_2\Bigg| & : u<s\leq 1.
  \end{array}
\right.
$$
Now notice that the reverse triangle inequality implies
\begin{equation}\label{eq:piece_bound}
\left|\om_{x,y}^u(s)-\left(x-\frac{Y_s^{(t)}}{f(t)}\right)\right|\leq\left\{
  \begin{array}{ll}
    \frac{1}{f(t)}\left\|\left(\bk_s^{(1)},\bk_s^{(2)},\bk_s^{(3)}\right)\right\|_2 & : 0\leq s\leq u\\
    \\
    \frac{1}{f(t)}\left\|\left(\bk_{s-u}^{(4)},\bk_{s-u}^{(5)},\bk_{s-u}^{(6)}\right)\right\|_2 & : u<s\leq 1.
  \end{array}
\right.
\end{equation}
Suppose $F$ has Lipschitz constant $K$. Then it follows from \eqref{eq:piece_bound} and subadditivity of the square root function that
\begin{equation}\label{eq:L3_Lip}
\Bigg|F\left(\om_{x,y}^u(\bullet)\right)-E\left[F\left(x-\frac{Y_\bullet^{(t)}}{f(t)}\right)\right]\Bigg|\leq\frac{K}{f(t)}\sum_{i=1}^6 E\Bigg[\left\|\bk_\bullet^{(i)}\right\|\Bigg].
\end{equation}
This bound is uniform in $x$ and $y$ but has an implicit dependence on $u$. We can easily remedy this situation by noting that Brownian scaling implies that the expected value of the uniform norm of a Brownian bridge from $0$ to $0$ of length $u$ is an increasing function of $u$. Hence we can write
$$\sum_{i=1}^6E\Bigg[\left\|\bk_\bullet^{(i)}\right\|\Bigg]\leq 6E\Big[\|\bk_\bullet\|\Big]$$
where $\bk$ is a standard Brownian bridge from $0$ to $0$ of length $1$.
This leads to a version of \eqref{eq:L3_Lip} which is uniform on $\{(x,y,u)\in \R^3:x>0,y>0,0<u<1\}$, namely
$$\Bigg|F\left(\om_{x,y}^u(\bullet)\right)-E\left[F\left(x-\frac{Y_\bullet^{(t)}}{f(t)}\right)\right]\Bigg|\leq\frac{6K}{f(t)}E\Big[\|\bk_\bullet\|\Big].$$
Together with \eqref{eq:L3_bridges} this proves the lemma since $f(t)\to\infty$ as $t\to\infty$.
\end{proof}

\begin{proof}[Proof of \eqref{eq:L3_scaled}]
Recalling that $\nu=-2h$ on $L_3$, Brownian scaling implies
\begin{equation}\label{eq:L3_h}
E_0\left[F\left(\frac{X_{\bullet t}}{t}\right)\exp\left(\nu S_t+hX_t\right)\right]=E_0\left[F\left(\frac{X_\bullet}{\rt{t}}\right)\exp\left(-h\rt{t}(2S_1-X_1)\right)\right].
\end{equation}
For $0<u<1$, define
$$f_t(x,y,u):=\left\{
  \begin{array}{ll}
    E_0\left[F\left(\frac{X_\bullet}{\rt{t}}\right)\middle|S_1=x\rt{t},S_1-X_1=y\rt{t},\Te_1=u\right] & : x,y>0\\
    \\
    0 & : \text{otherwise}.
  \end{array}
\right.$$
Using the tri-variate density \eqref{eq:tri_var}, the right-hand side of \eqref{eq:L3_h} can be written as
$$\int_0^1\int_0^\infty\int_0^\infty\frac{xy~f_t\left(\frac{x}{\rt{t}},\frac{y}{\rt{t}},u\right)}{\pi\rt{u^3(1-u)^3}}\exp\left(-h\rt{t}(x+y)-\frac{x^2}{2u}-\frac{y^2}{2(1-u)}\right)\dd{x}\dd{y}\dd{u}.$$
Applying the change of variables $x\mapsto x\rt{u}-uh\rt{t}$ and $y\mapsto y\rt{1-u}-(1-u)h\rt{t}$ results in
\begin{equation}\label{eq:L3_vars}
\int_0^1\int_{h\rt{ut}}^\infty\int_{h\rt{(1-u)t}}^\infty\tilde{f}_t(x,y,u)~g_t(x,y,u)\frac{h^2 t}{\pi}\exp\left(-\frac{x^2+y^2}{2}+\frac{1}{2}h^2 t\right)\dd{x}\dd{y}\dd{u}
\end{equation}
where we defined
$$\tilde{f}_t(x,y,u):=f_t\left(\frac{x\rt{u}}{\rt{t}}-uh,\frac{y\rt{1-u}}{\rt{t}}-(1-u)h,u\right)$$
and
$$g_t(x,y,u):=\frac{\left(x-h\rt{ut}\right)\left(y-h\rt{(1-u)t}\right)}{h^2 t\rt{u(1-u)}}.$$
Now we divide \eqref{eq:L3_vars} by the $L_3$ partition function asymptotic from \autoref{prop:asymptotics} which gives
\begin{equation}\label{eq:L3_Fatou}
\int_0^1\int_{-\infty}^\infty\int_{-\infty}^\infty\tilde{f}_t(x,y,u)~g_t(x,y,u)~1_{A_t}(x,y)~\frac{1}{2\pi}\exp\left(-\frac{x^2+y^2}{2}\right)\dd{x}\dd{y}\dd{u}
\end{equation}
where we defined
$$A_t:=\left\{(x,y):x>h\rt{ut},y>h\rt{(1-u)t}\right\}.$$
At this stage we want to to find the limit of \eqref{eq:L3_Fatou} as $t\to\infty$ by appealing to \autoref{lem:Fatou} with $\mu$ being the probability measure on $\R\times\R\times [0,1]$ having density $\frac{1}{2\pi}\exp\left(-\frac{x^2+y^2}{2}\right)$. In this direction, note that $\tilde{f}_t$ is bounded and the fact that the convergence in \autoref{lem:uniform} is uniform and $(x,y,u)\mapsto F\left(\om_{x,y}^u(\bullet)\right)$ is continuous on $\{(x,y,u)\in \R^3:x>0,y>0,0<u<1\}$ implies that 
\begin{equation*}
\begin{split}  
\lim_{t\to\infty}\tilde{f}_t(x,y,u)&=F\left(\om_{-hu,-h(1-u)}^u(\bullet)\right)\\
																													 &=F\left(-h\uk^u_\bullet\right)
\end{split}
\end{equation*}
$\mu$-almost surely. Additionally, $g_t 1_{A_t}$ is non-negative and converges $\mu$-almost surely to $1$ as $t\to\infty$. Lastly, by reversing the steps that led from \eqref{eq:L3_h} to \eqref{eq:L3_Fatou}, we see that
\begin{equation*}
\begin{split}
\lim_{t\to\infty}\int_0^1\int_{-\infty}^\infty\int_{-\infty}^\infty g_t(x,y,u)~1_{A_t}(x,y)~\frac{1}{2\pi}\exp\left(-\frac{x^2+y^2}{2}\right)\dd{x}\dd{y}\dd{u}&=\lim_{t\to\infty}\frac{E_0\left[\exp\left(\nu S_t+hX_t\right)\right]}{2h^2 t\exp\left(\frac{1}{2}h^2 t\right)}\\
&=1.
\end{split}
\end{equation*}
Hence by \autoref{lem:Fatou} we can conclude that 
\begin{equation*}
\begin{split}
\lim_{t\to\infty}Q_t^{\nu,h}\left[F\left(\frac{X_{\bullet t}}{t}\right)\right]&=\int_0^1\int_{-\infty}^\infty\int_{-\infty}^\infty F\left(-h\uk_\bullet^u\right)\frac{1}{2\pi}\exp\left(-\frac{x^2+y^2}{2}\right)\dd{x}\dd{y}\dd{u}\\
&=\int_0^1 F\left(-h\uk_\bullet^u\right)\dd{u}\\
&=E\left[F\left(-h\uk_\bullet^U\right)\right].
\end{split}
\end{equation*}
\end{proof}

\subsection[R1 case]{$R_1$ case}\label{sec:R_1}

In this section we prove the $R_1$ row in the table from \autoref{thm:theorem_1}. That the limit under $Q^{\nu,h}|_{\mathcal{F}_t}$ is $-R$ follows from \autoref{thmL:2parameters} and \autoref{thmL:scaled} since $Q^{\nu,h}=Q^{\nu+h}$ with $\nu+h<0$ when $(\nu,h)\in R_1$. To prove the limit under $Q_t^{\nu,h}|_{\mathcal{F}_t}$ is $-\ek$, we show that 
\begin{equation}\label{eq:R_1}
\lim_{t\to\infty}Q_t^{\nu,h}\left[F\left(\frac{X_{\bullet t}}{\rt{t}}\right)\right]=E\left[F(-\ek_\bullet)\right]
\end{equation}
for any bounded continuous $F:\mathcal{C}\big([0,1];\R\big)\to\R$ and $(\nu,h)\in R_1$.

\begin{proof}[Proof of \eqref{eq:R_1}]
From Brownian scaling, we have
\begin{equation}\label{eq:R1_scaled}
E_0\left[F\left(\frac{X_{\bullet t}}{\rt{t}}\right)\exp\left(\nu S_t+hX_t\right)\right]=E_0\left[F\left(X_\bullet\right)\exp\left(\nu\rt{t}S_1+h\rt{t}X_1\right)\right].
\end{equation}
Since $\nu<0$, we can write
$$\exp\left(\nu\rt{t}S_1\right)=-\nu\rt{t}\int_0^\infty\exp\left(\nu\rt{t}x\right)1_{S_1<x}\dd{x}.$$
Hence the right-hand side of \eqref{eq:R1_scaled} can be written as
\begin{align}
-\nu&\rt{t}\int_0^\infty\exp\left(\nu\rt{t}x\right)E_0\left[F\left(X_\bullet\right)\exp\left(h\rt{t}X_1\right);S_1<x\right]\dd{x}\nonumber\\
&=-\nu\rt{t}\int_0^\infty\exp\left(\nu\rt{t}x\right)E_0\left[F\left(-X_\bullet\right)\exp\left(-h\rt{t}X_1\right);I_1>-x\right]\dd{x}\nonumber\\
\label{eq:R1_symmetries}
&=-\nu\rt{t}\int_0^\infty\exp\left((\nu+h)\rt{t}x\right)E_x\left[F\left(x-X_\bullet\right)\exp\left(-h\rt{t}X_1\right);I_1>0\right]\dd{x}
\end{align}
where the two equalities follow from the reflection and translation symmetries of Wiener measure.
The $h$-transform representation of the Bessel(3) path measure from \autoref{prop:h_abs} can be used to rewrite the expectation appearing in \eqref{eq:R1_symmetries} in terms of a Bessel(3) process $(R_s:0\leq s\leq 1)$. This leads to 
$$-\nu\rt{t}\int_0^\infty\exp\left((\nu+h)\rt{t}x\right)E_x\left[F\left(x-R_\bullet\right)\exp\left(-h\rt{t}R_1\right)\frac{x}{R_1}\right]\dd{x}.$$
Next we disintegrate the Bessel(3) path measure into a mixture of Bessel(3) bridge measures by conditioning on the endpoint of the path. Refer to Proposition 1 in \cite{bridges1} for a precise statement of a more general result. See also Theorem 1 in \cite{bridges3} where weak continuity of the bridge measures with respect to their starting and ending points is established. This results in 
\begin{equation}\label{eq:R1_disint}
-\nu\rt{t}\int_0^\infty\int_0^\infty\exp\left((\nu+h)\rt{t}x-h\rt{t}y\right)\frac{x}{y}E_x\left[F\left(x-R_\bullet\right)\middle|R_1=y\right]P_x(R_1\in\dd{y})\dd{x}.
\end{equation}
Using the Bessel(3) transition density formula \eqref{eq:Besx_dens}, we can now write \eqref{eq:R1_disint} as
$$-\nu\rt{\frac{2t}{\pi}}\int_0^\infty\int_0^\infty\exp\left((\nu+h)\rt{t}x-h\rt{t}y\right)f(x,y)\sinh(xy)\exp\left(-\frac{x^2+y^2}{2}\right)\dd{y}\dd{x}$$
where we defined 
$$f(x,y):=E_x\left[F\left(x-R_\bullet\right)\middle|R_1=y\right].$$
Applying the change of variables $x\mapsto x/\rt{t}$ and $y\mapsto y/\rt{t}$ gives
\begin{equation}\label{eq:gamma}
-\frac{\nu}{h^2(\nu+h)^2}\rt{\frac{2}{\pi t}}\int_0^\infty\int_0^\infty g(x,y)f\left(\frac{x}{\rt{t}},\frac{y}{\rt{t}}\right)\frac{1}{xy}\sinh\left(\frac{xy}{t}\right)\exp\left(-\frac{x^2+y^2}{2t}\right)\dd{y}\dd{x}
\end{equation}
where we defined
$$g(x,y):=h^2(\nu+h)^2 xy\exp\big((\nu+h)x-hy\big).$$
After dividing \eqref{eq:gamma} by the $R_1$ partition function asymptotic from \autoref{prop:asymptotics}, we see that showing
\begin{equation}\label{eq:R1_limit}
\lim_{t\to\infty}\int_0^\infty\int_0^\infty g(x,y)f\left(\frac{x}{\rt{t}},\frac{y}{\rt{t}}\right)\frac{t}{xy}\sinh\left(\frac{xy}{t}\right)\exp\left(-\frac{x^2+y^2}{2t}\right)\dd{y}\dd{x}=E\left[F(-\ek_\bullet)\right]
\end{equation}
will prove \eqref{eq:R_1}. Notice that for all $x,y>0$, the limit
$$\lim_{t\to\infty}f\left(\frac{x}{\rt{t}},\frac{y}{\rt{t}}\right)\frac{t}{xy}\sinh\left(\frac{xy}{t}\right)\exp\left(-\frac{x^2+y^2}{2t}\right)=E\left[F(-\ek_\bullet)\right]$$
follows from the weak continuity of the bridge measures with respect to their starting and ending points which was noted above and the fact that a normalized Brownian excursion is simply a Bessel(3) bridge from $0$ to $0$ of unit length. Additionally, the convexity of $\sinh$ on $[0,1]$ along with the inequality $2xy\leq x^2+y^2$ leads to the bound
$$\left|f\left(\frac{x}{\rt{t}},\frac{y}{\rt{t}}\right)\frac{t}{xy}\sinh\left(\frac{xy}{t}\right)\exp\left(-\frac{x^2+y^2}{2t}\right)\right|\leq\|F\|\sinh(1)$$
which holds for all $x,y,t>0$. Noting that $g$ is a probability density, \eqref{eq:R1_limit} now follows from bounded convergence.
\end{proof}

\section[Proof of Theorem 1.2]{Proof of \autoref{thm:theorem_2}}\label{sec:theorem_2}

In this section we prove \autoref{thm:theorem_2} by showing that the following limits hold for any bounded continuous $F:\mathcal{C}\big([0,1];\R\big)\to\R$:
\begin{equation}\label{eq:FCLT_Q}
\text{if } (\nu,h)\in R_2\cup L_3\text{, then }\lim_{t\to\infty}Q^{\nu,h}\left[F\left(\frac{X_{\bullet t}-(\nu+h)t\uk_\bullet}{\rt{t}}\right)\right]=E_0\left[F(X_\bullet)\right],
\end{equation}
\vs
\begin{equation}\label{eq:FCLT_Qdelta}
\text{if } (\nu,h)\in R_3\text{, then }\lim_{t\to\infty}Q^{\nu,h}\left[F\left(\frac{X_{\bullet t}-ht\uk_\bullet}{\rt{t}}\right)\right]=E_0\left[F(X_\bullet)\right],
\end{equation}
\vs
\begin{equation}\label{eq:FCLT_scaled}
\left.
\begin{array}{ll}
\text{if } (\nu,h)\in R_2\text{, then }
\displaystyle\lim_{t\to\infty}Q_t^{\nu,h}\left[F\left(\frac{X_{\bullet t}-(\nu+h)t\uk_\bullet}{\rt{t}}\right)\right]\\
\\
\text{if } (\nu,h)\in R_3\text{, then }
\displaystyle\lim_{t\to\infty}Q_t^{\nu,h}\left[F\left(\frac{X_{\bullet t}-ht\uk_\bullet}{\rt{t}}\right)\right]\\
\end{array}
\right\}
=E_0\left[F(X_\bullet)\right].
\end{equation}

\begin{proof}[Proof of \eqref{eq:FCLT_Q}]
If $(\nu,h)\in R_2\cup L_3$, we can use \autoref{thmL:2parameters}, a path transformation that adds drift $\nu+h$, and Brownian scaling to write 
\begin{equation*}
\begin{split}
Q^{\nu,h}\left[F\left(\frac{X_{\bullet t}-(\nu+h)t\uk_\bullet}{\rt{t}}\right)\right]&=E_0^{\nu+h}\left[F\left(\frac{X_{\bullet t}-(\nu+h)t\uk_\bullet}{\rt{t}}\right)\right]\\
&=E_0\left[F\left(\frac{X_{\bullet t}}{\rt{t}}\right)\right]\\
&=E_0\left[F\left(X_\bullet\right)\right]
\end{split}
\end{equation*}
from which the desired limit follows. 
\end{proof}

An idea that will be helpful for the proof of \eqref{eq:FCLT_Qdelta} and also in the next section is to modify the path functional $F$ in such a way so that it ``ignores" the beginning of the path. After proving a limit theorem for the modified path functional, we lift this result to the original functional by controlling the error arising from the modification. Here we state some definitions and notation that make this procedure precise. 
\begin{definition}\label{def:F_delta}
For $0<\de\leq 1$ and $X_\bullet\in \mathcal{C}\big([0,1];\R\big)$, define $\Pi_\de:\mathcal{C}\big([0,1];\R\big)\to \mathcal{C}_0\big([0,1];\R\big)$ by 
$$\Pi_\de X_s = \left\{
  \begin{array}{ll}
    \frac{X_\de}{\de}s & : 0\leq s<\de\\
    \\
    X_s & : \de\leq s\leq 1.
  \end{array}
  \right.
$$
For $F:\mathcal{C}\big([0,1];\R\big)\to\R$, let $F_\de$ denote $F\circ\Pi_\de$ and define $\De_\de^F:\mathcal{C}\big([0,1];\R\big)\to\R_+$ by
$$\De_\de^F(X_\bullet)=|F_\de(X_\bullet)-F(X_\bullet)|.$$
\end{definition}
\no So $\Pi_\de$ replaces the initial $[0,\de]$ segment of the path $X_\bullet$ with a straight line which interpolates between the points $(0,0)$ and $(\de,X_\de)$ while $\De_\de^F$ is the absolute error that results from using $F_\de$ instead of $F$. The Markov property implies that under $P_0$, the random variable $F_\de(X_\bullet)$ is independent of the initial $[0,\de]$ part of the path $X_\bullet$ after conditioning on $X_\de$. Note that if $F$ is bounded continuous then so are $F_\de$ and $\De_\de^F$ with $\|F_\de\|\leq\|F\|$ and $\|\De_\de^F\|\leq 2\|F\|$. Also notice that $\displaystyle\lim_{\de\searrow 0}\De_\de^F(X_\bullet)=0$ for any $X_\bullet\in \mathcal{C}_0\big([0,1];\R\big)$ whenever $F$ is continuous.

\begin{proof}[Proof of \eqref{eq:FCLT_Qdelta}]
We divide the proof into two stages, the first for $F_\de$ and the second for $F$.\vs
\begin{enumerate}[label=\textbf{\arabic*.}]
\item\label{itm:stage_1}
\no\textbf{\textit{Convergence for $\boldsymbol{F_\de}$}}
\vs

Suppose $(\nu,h)\in R_3$ and fix $0<\de\leq 1$. In this case \autoref{thmL:2parameters} implies 
\begin{equation}\label{eq:R3_QP}
Q^{\nu,h}\left[F_\de\left(\frac{X_{\bullet t}-ht\uk_\bullet}{\rt{t}}\right)\right]=\frac{\nu+2h}{2h}E_0^h\left[F_\de\left(\frac{X_{\bullet t}-ht\uk_\bullet}{\rt{t}}\right)\exp(\nu S_\infty)\right].
\end{equation}
We can rewrite the expectation appearing on the right-hand side of \eqref{eq:R3_QP} as
$$\underbrace{E_0^h\left[F_\de\left(\frac{X_{\bullet t}-ht\uk_\bullet}{\rt{t}}\right)\exp(\nu S_{\de t});\Te_\infty\leq \de t\right]}_{A_t}+\underbrace{E_0^h\left[F_\de\left(\frac{X_{\bullet t}-ht\uk_\bullet}{\rt{t}}\right)\exp(\nu S_\infty);\Te_\infty>\de t\right]}_{B_t} .$$
Since $(\nu,h)\in R_3$, we know from Williams' decomposition that $\exp(\nu S_\infty)$ is integrable and ${\Te_\infty<\infty}$ almost surely under $P_0^h$. Hence by dominated convergence we have $B_t=o(1)$ as $t\to\infty$ and consequently
$$A_t=E_0^h\left[F_\de\left(\frac{X_{\bullet t}-ht\uk_\bullet}{\rt{t}}\right)\exp(\nu S_\infty)\right]+o(1)\text{ as }t\to \infty.$$
Similarly, we can show that
$$A_t=E_0^h\left[F_\de\left(\frac{X_{\bullet t}-ht\uk_\bullet}{\rt{t}}\right)\exp(\nu S_{\de t})\right]+o(1)\text{ as }t\to \infty.$$
Together, these imply
\begin{equation}\label{eq:R3_uniform}
E_0^h\left[F_\de\left(\frac{X_{\bullet t}-ht\uk_\bullet}{\rt{t}}\right)\exp(\nu S_\infty)\right]=E_0^h\left[F_\de\left(\frac{X_{\bullet t}-ht\uk_\bullet}{\rt{t}}\right)\exp(\nu S_{\de t})\right]+o(1)\text{ as }t\to \infty.
\end{equation}
Applying a path transformation that adds drift $h$ on the right-hand side of \eqref{eq:R3_uniform} and using Brownian scaling results in
$$\resizebox{.9\hsize}{!}{$\displaystyle E_0^h\left[F_\de\left(\frac{X_{\bullet t}-ht\uk_\bullet}{\rt{t}}\right)\exp(\nu S_\infty)\right]=E_0\left[F_\de\left(X_\bullet\right)\exp\left(\nu\rt{t}\sup_{0\leq s\leq \de}\left\{X_s+h\rt{t}s\right\}\right)\right]+o(1)\text{ as }t\to \infty.$}$$
Combining this with \eqref{eq:R3_QP}, we have established that
\begin{equation}\label{eq:R3_Q}
\resizebox{.85\hsize}{!}{$\displaystyle Q^{\nu,h}\left[F_\de\left(\frac{X_{\bullet t}-ht\uk_\bullet}{\rt{t}}\right)\right]=\frac{\nu+2h}{2h}E_0\left[F_\de\left(X_\bullet\right)\exp\left(\nu\rt{t}\sup_{0\leq s\leq \de}\left\{X_s+h\rt{t}s\right\}\right)\right]+o(1)\text{ as }t\to \infty.$}
\end{equation}
Now notice that $F_\de\left(X_\bullet\right)$ and $\displaystyle\sup_{0\leq s\leq\de}\left\{X_s+h\rt{t}s\right\}$ are independent after conditioning on $X_\de$. So with $p_\de(\cdot,\cdot)$ denoting the transition density of Brownian motion at time $\de$, we see that the expectation appearing on the right-hand side of \eqref{eq:R3_Q} is equal to
\begin{equation}\label{eq:R3_decouple}
\int_{-\infty}^\infty E_0\left[F_\de(X_\bullet)\middle|X_\de=x\right]E_0\left[\exp\left(\nu\rt{t}\sup_{0\leq s\leq\de}\left\{X_s+h\rt{t}s\right\}\right)\middle|X_\de=x\right]p_\de(0,x)\dd{x}.
\end{equation}
From the particular pathwise construction of Brownian bridge given in (5.6.29) of \cite{K&S}, it follows that the distribution of Brownian bridge plus a constant drift is the same as Brownian bridge with an appropriately shifted endpoint. Hence the second expectation appearing inside the integral in \eqref{eq:R3_decouple} is seen to equal 
$$E_0\left[\exp\left(\nu\rt{t}\sup_{0\leq s\leq \de}X_s\right)\middle|X_\de=x+h\rt{t}\de\right].$$
By using the distribution of the maximum of a Brownian bridge from \eqref{eq:bridge_max}, this expectation has the integral representation 
$$\int_0^\infty \frac{4y-2x-2h\rt{t}\de}{\de}\exp\left(\nu\rt{t}y-\frac{2y(y-x-h\rt{t}\de)}{\de}\right)\dd{y}$$
when $t$ is large enough such that $x+h\rt{t}\de\leq 0$.
After some manipulations and recalling that $\nu+2h<0$, we can use Watson's lemma to compute the limit of this integral which holds for all $x\in\R$:
$$\lim_{t\to\infty}\int_0^\infty \left(\frac{4y-2x}{\de}-2h\rt{t}\right)\exp\left(\frac{2yx-2y^2}{\de}\right)e^{(\nu+2h)\rt{t}y}\dd{y}=\frac{2h}{\nu+2h}.$$
Now we want to invoke \autoref{lem:Fatou} to find the limit of \eqref{eq:R3_decouple} hence we need to verify that 
\begin{equation}\label{eq:R3_check}
\lim_{t\to\infty}\int_{-\infty}^\infty E_0\left[\exp\left(\nu\rt{t}\sup_{0\leq s\leq\de}\left\{X_s+h\rt{t}s\right\}\right)\middle|X_\de=x\right]p_\de(0,x)\dd{x}=\frac{2h}{\nu+2h}.
\end{equation}
By working backwards starting from the left-hand side of \eqref{eq:R3_check} and reversing the conditioning, scaling, and path transformation, this can be reduced to checking 
$$\lim_{t\to\infty}E_0^h\left[\exp\left(\nu S_{\de t}\right)\right]=\frac{2h}{\nu+2h}$$
which follows from dominated convergence and Williams' path decomposition since $(\nu,h)\in R_3$. Now we can evaluate the limit of \eqref{eq:R3_decouple} as
\begin{equation*}
\begin{split}
\lim_{t\to\infty}&\int_{-\infty}^\infty E_0\left[F_\de(X_\bullet)\middle|X_\de=x\right]E_0\left[\exp\left(\nu\rt{t}\sup_{0\leq s\leq\de}\left\{X_s+h\rt{t}s\right\}\right)\middle|X_\de=x\right]p_\de(0,x)\dd{x}\\
&=\frac{2h}{\nu+2h}\int_{-\infty}^\infty E_0\left[F_\de(X_\bullet)\middle|X_\de=x\right]p_\de(0,x)\dd{x}\\
&=\frac{2h}{\nu+2h}E_0\left[F_\de(X_\bullet)\right].
\end{split}
\end{equation*}
Combining this with \eqref{eq:R3_Q} leads to
$$\lim_{t\to\infty}Q^{\nu,h}\left[F_\de\left(\frac{X_{\bullet t}-ht\uk_\bullet}{\rt{t}}\right)\right]=E_0\left[F_\de(X_\bullet)\right]$$
as desired.
\vs
\item\label{itm:stage_2}
\no\textbf{\textit{Convergence for $\boldsymbol{F}$}}
\vs

Recall the notation $\De_\de^F$ from \autoref{def:F_delta}. Using the triangle inequality, we can write for all $t$
\begin{equation*}
\begin{split}
&\left|Q^{\nu,h}\left[F\left(\frac{X_{\bullet t}-ht\uk_\bullet}{\rt{t}}\right)\right]-E_0\left[F(X_\bullet)\right]\right|\\
&\leq Q^{\nu,h}\left[\De_\de^F\left(\frac{X_{\bullet t}-ht\uk_\bullet}{\rt{t}}\right)\right]+\Bigg|Q^{\nu,h}\left[F_\de\left(\frac{X_{\bullet t}-ht\uk_\bullet}{\rt{t}}\right)\right]-E_0\left[F_\de(X_\bullet)\right]\Bigg|+E_0\left[\De_\de^F(X_\bullet)\right].
\end{split}
\end{equation*}
The result from stage \ref{itm:stage_1} implies that the middle term on the right-hand side of this inequality vanishes as $t\to\infty$ so
$$\limsup_{t\to\infty}\left|Q^{\nu,h}\left[F\left(\frac{X_{\bullet t}-ht\uk_\bullet}{\rt{t}}\right)\right]-E_0\left[F(X_\bullet)\right]\right|\leq \limsup_{t\to\infty} Q^{\nu,h}\left[\De_\de^F\left(\frac{X_{\bullet t}-ht\uk_\bullet}{\rt{t}}\right)\right]+E_0\left[\De_\de^F(X_\bullet)\right].$$
Since $F$ is bounded continuous, we know that the last term on the right-hand side of this inequality vanishes as $\de\searrow 0$ by bounded convergence. This leads to
\begin{equation}\label{eq:delta_error}
\limsup_{t\to\infty}\left|Q^{\nu,h}\left[F\left(\frac{X_{\bullet t}-ht\uk_\bullet}{\rt{t}}\right)\right]-E_0\left[F(X_\bullet)\right]\right|\leq \lim_{\de\searrow 0}\limsup_{t\to\infty} Q^{\nu,h}\left[\De_\de^F\left(\frac{X_{\bullet t}-ht\uk_\bullet}{\rt{t}}\right)\right].
\end{equation}
Using \autoref{thmL:2parameters}, we can express the right-hand side of \eqref{eq:delta_error} as
\begin{equation}\label{eq:error_bound}
\frac{\nu+2h}{2h}\lim_{\de\searrow 0}\limsup_{t\to\infty}E_0^h\left[\De_\de^F\left(\frac{X_{\bullet t}-ht\uk_\bullet}{\rt{t}}\right)\exp(\nu S_\infty)\right].
\end{equation}
Since $h<-\frac{1}{2}\nu$ when $(\nu,h)\in R_3$, we can find $p>1$ such that $-2h>p\nu$. Let $q$ be the H\"{o}lder conjugate of $p$. Then H\"{o}lder's inequality implies \eqref{eq:error_bound} is bounded above by
\begin{equation*}
\begin{split}
\frac{\nu+2h}{2h}&\lim_{\de\searrow 0}\limsup_{t\to\infty}E_0^h\left[\left(\De_\de^F\left(\frac{X_{\bullet t}-ht\uk_\bullet}{\rt{t}}\right)\right)^q\right]^\frac{1}{q}E_0^h[\exp(p\nu S_\infty)]^\frac{1}{p}\\
&=\frac{\nu+2h}{2h}\lim_{\de\searrow 0}E_0\left[\left(\De_\de^F\left(X_\bullet\right)\right)^q\right]^\frac{1}{q}\left(\frac{2h}{pv+2h}\right)^\frac{1}{p}\\
&=0.
\end{split}
\end{equation*}
Here we used a path transformation that adds drift $h$ along with Brownian scaling to eliminate $t$ from the first expectation and used Williams' decomposition to compute the second expectation. Now it follows that 
$$\limsup_{t\to\infty}\left|Q^{\nu,h}\left[F\left(\frac{X_{\bullet t}-ht\uk_\bullet}{\rt{t}}\right)\right]-E_0\left[F(X_\bullet)\right]\right|=0$$
which proves \eqref{eq:FCLT_Qdelta}.
\end{enumerate}
\end{proof}

\begin{proof}[Proof of \eqref{eq:FCLT_scaled}]
We first show that the limit holds in the $R_3$ case and then use duality to transfer this result to the $R_2$ case. Accordingly, suppose $(\nu,h)\in R_3$. Using a Girsanov change of measure, we can write
\begin{equation}\label{eq:FCLT_Girs}
E_0\left[F\left(\frac{X_{\bullet t}-ht\uk_\bullet}{\rt{t}}\right)\exp\left(\nu S_t+hX_t\right)\right]=\exp\left(\frac{1}{2}h^2 t\right)E_0^h\left[F\left(\frac{X_{\bullet t}-ht\uk_\bullet}{\rt{t}}\right)\exp\left(\nu S_t\right)\right].
\end{equation}
By repeating the same argument that led to \eqref{eq:R3_uniform} while using $F$ instead of $F_\de$, we can establish that 
$$E_0^h\left[F\left(\frac{X_{\bullet t}-ht\uk_\bullet}{\rt{t}}\right)\exp(\nu S_t)\right]=E_0^h\left[F\left(\frac{X_{\bullet t}-ht\uk_\bullet}{\rt{t}}\right)\exp(\nu S_\infty)\right]+o(1)\text{ as }t\to \infty.$$
Combining this with \eqref{eq:FCLT_Girs} and using the partition function asymptotic from \autoref{prop:asymptotics} results in 
\begin{equation*}
\begin{split}
\lim_{t\to\infty}Q_t^{\nu,h}\left[F\left(\frac{X_{\bullet t}-ht\uk_\bullet}{\rt{t}}\right)\right]&=\lim_{t\to\infty}\frac{\nu+2h}{2h}E_0^h\left[F\left(\frac{X_{\bullet t}-ht\uk_\bullet}{\rt{t}}\right)\exp(\nu S_\infty)\right]\\
&=\lim_{t\to\infty}Q^{\nu,h}\left[F\left(\frac{X_{\bullet t}-ht\uk_\bullet}{\rt{t}}\right)\right]\\
&=E_0[F(X_\bullet)]
\end{split}
\end{equation*}
where the last two equalities follow from \autoref{thmL:2parameters} and \eqref{eq:FCLT_Qdelta}, respectively.

Now suppose $(\nu,h)\in R_2$. Then $\big(\nu,-(\nu+h)\big)\in R_3$. Hence the invariance property of $\phi$ and the above result imply
\begin{equation*}
\begin{split}
\lim_{t\to\infty}Q_t^{\nu,h}\left[F\left(\frac{X_{\bullet t}-ht\uk_\bullet}{\rt{t}}\right)\right]&=\lim_{t\to\infty}Q_t^{\nu,-(\nu+h)}\left[F_\phi\left(\frac{X_{\bullet t}-ht\uk_\bullet}{\rt{t}}\right)\right]\\
&=E_0[F_\phi(X_\bullet)]\\
&=E_0[F(X_\bullet)].
\end{split}
\end{equation*}
\end{proof}

\section[Proof of Theorem 1.3]{Proof of \autoref{thm:theorem_3}}\label{sec:theorem_3}

In this section we prove \autoref{thm:theorem_3} by showing that 
\begin{equation}\label{eq:L3_FCLT}
\lim_{t\to\infty}Q_t^{\nu,h}\left[F\left(\frac{X_{\bullet t}-(2S_{\bullet t}+ht\uk_\bullet)}{\rt{t}}\right)\right]=E_0\left[F(X_\bullet)\right]
\end{equation}
for any bounded Lipschitz continuous $F:\mathcal{C}\big([0,1];\R\big)\to\R$ and $(\nu,h)\in L_3$. 

\begin{proof}[Proof of \eqref{eq:L3_FCLT}]
We divide the proof into two stages, the first for $F_\de$ and the second for $F$.\vs
\begin{enumerate}[label=\textbf{\arabic*.}]
\item
\no\textbf{\textit{Convergence for $\boldsymbol{F_\de}$}}
\vs

Suppose $(\nu,h)\in L_3$ and fix $0<\de\leq 1$. Recalling that $\nu=-2h$, we can use Brownian scaling and Pitman's $2S-X$ theorem to write
\begin{equation}\label{eq:L3_Bessel}
E_0\left[F_\de\left(\frac{X_{\bullet t}-(2S_{\bullet t}+ht\uk_\bullet)}{\rt{t}}\right)\exp\left(\nu S_t+hX_t\right)\right]=E_0\left[F_\de\left(-R_\bullet-h\rt{t}\uk_\bullet\right)\exp\left(-h\rt{t}R_1\right)\right].
\end{equation}
Fix $x>0$. The the absolute continuity relation from \autoref{lem:absolute} implies that the right-hand side of \eqref{eq:L3_Bessel} is equal to 
$$E_x\left[F_\de\left(-R_\bullet-h\rt{t}\uk_\bullet\right)\exp\left(-h\rt{t}R_1\right)\frac{xR_\de\exp\left(\frac{x^2}{2\de}\right)}{\de\sinh\left(\frac{xR_\de}{\de}\right)}\right].$$
Now we can use \autoref{prop:h_abs} to switch from the Bessel(3) process to Brownian motion. Hence the above expectation is equal to
$$E_x\left[F_\de\left(-X_\bullet-h\rt{t}\uk_\bullet\right)\exp\left(-h\rt{t}X_1\right)\frac{X_\de X_1\exp\left(\frac{x^2}{2\de}\right)}{\de\sinh\left(\frac{xX_\de}{\de}\right)};I_1>0\right].$$
Next we use a Girsanov change of measure to add drift $-h\rt{t}$. This results in
$$E_x^{-h\rt{t}}\left[F_\de\left(-X_\bullet-h\rt{t}\uk_\bullet\right)\exp\left(-h\rt{t}x+\frac{1}{2}h^2 t\right)\frac{X_\de X_1\exp\left(\frac{x^2}{2\de}\right)}{\de\sinh\left(\frac{xX_\de}{\de}\right)};I_1>0\right].$$
Applying a path transformation that adds drift $-h\rt{t}$ while changing the measure back to that of Brownian motion without drift yields
$$E_x\left[F_\de\left(-X_\bullet\right)\frac{(X_\de-h\rt{t}\de)(X_1-h\rt{t})\exp\left(\frac{x^2}{2\de}-h\rt{t}x+\frac{1}{2}h^2 t\right)}{\de\sinh\left(\frac{x}{\de}(X_\de-h\rt{t}\de)\right)};\inf_{0\leq s\leq 1}\{X_s-h\rt{t}s\}>0\right].$$
Now we divide this by the $L_3$ partition function asymptotic from \autoref{prop:asymptotics} which gives
\begin{equation}\label{eq:L3}
E_x\left[F_\de\left(-X_\bullet\right)\frac{(X_\de-h\rt{t}\de)(X_1-h\rt{t})\exp\left(\frac{x^2}{2\de}-h\rt{t}x\right)}{2h^2 t\de\sinh\left(\frac{x}{\de}(X_\de-h\rt{t}\de)\right)}1_{A_t}\right]
\end{equation}
where we defined
$$A_t:=\left\{\inf_{0\leq s\leq 1}\{X_s-h\rt{t}s\}>0\right\}.$$
At this point we want to use \autoref{lem:Fatou} to find the limit of \eqref{eq:L3} as $t\to\infty$. Towards this end, note that
$$\frac{(X_\de-h\rt{t}\de)(X_1-h\rt{t})\exp\left(\frac{x^2}{2\de}-h\rt{t}x\right)}{2h^2 t\de\sinh\left(\frac{x}{\de}(X_\de-h\rt{t}\de)\right)}1_{A_t}$$
is non-negative for all $t>0$. Recalling that $x>0$ and $h<0$, we see that almost surely under $P_x$
$$\lim_{t\to\infty}\frac{(X_\de-h\rt{t}\de)(X_1-h\rt{t})\exp\left(\frac{x^2}{2\de}-h\rt{t}x\right)}{2h^2 t\de\sinh\left(\frac{x}{\de}(X_\de-h\rt{t}\de)\right)}1_{A_t}=\exp\left(\frac{x^2-2xX_\de}{2\de}\right).$$
Additionally, by reversing the steps that led from \eqref{eq:L3_Bessel} to \eqref{eq:L3}, we see that    
$$\lim_{t\to\infty}E_x\left[\frac{(X_\de-h\rt{t}\de)(X_1-h\rt{t})\exp\left(\frac{x^2}{2\de}-h\rt{t}x\right)}{2h^2 t\de\sinh\left(\frac{x}{\de}(X_\de-h\rt{t}\de)\right)}1_{A_t}\right]=\lim_{t\to\infty}\frac{E_0\left[\exp\left(\nu S_t+hX_t\right)\right]}{2h^2 t\exp\left(\frac{1}{2}h^2 t\right)}=1.$$
This agrees with
$$E_x\left[\exp\left(\frac{x^2-2xX_\de}{2\de}\right)\right]=1$$
which follows from a routine calculation. Hence we can conclude from \autoref{lem:Fatou}, reflection symmetry of Wiener measure, and \autoref{lem:absolute} that  
\begin{equation*}
\begin{split}
\lim_{t\to\infty}&E_x\left[F_\de\left(-X_\bullet\right)\frac{(X_\de-h\rt{t}\de)(X_1-h\rt{t})\exp\left(\frac{x^2}{2\de}-h\rt{t}x\right)}{2h^2 t\de\sinh\left(\frac{x}{\de}(X_\de-h\rt{t}\de)\right)};\inf_{0\leq s\leq 1}\{X_s-h\rt{t}s\}>0\right]\\
&=E_x\left[F_\de\left(-X_\bullet\right)\exp\left(\frac{x^2-2xX_\de}{2\de}\right)\right]\\
&=E_{-x}\left[F_\de\left(X_\bullet\right)\exp\left(\frac{x^2+2xX_\de}{2\de}\right)\right]\\
&=E_0\left[F_\de\left(X_\bullet\right)\right].
\end{split}
\end{equation*}
This shows that for any $0<\de\leq 1$ we have
$$\lim_{t\to\infty}Q_t^{\nu,h}\left[F_\de\left(\frac{X_{\bullet t}-(2S_{\bullet t}+ht\uk_\bullet)}{\rt{t}}\right)\right]=E_0\left[F_\de(X_\bullet)\right].$$
\vs

\item
\no\textbf{\textit{Convergence for $\boldsymbol{F}$}}
\vs

We proceed as in the beginning of stage \ref{itm:stage_2} in the proof of \eqref{eq:FCLT_Qdelta}. Similarly to \eqref{eq:delta_error} we have
\begin{align}
\limsup_{t\to\infty}&\Bigg|Q_t^{\nu,h}\left[F\left(\frac{X_{\bullet t}-(2S_{\bullet t}+ht\uk_\bullet)}{\rt{t}}\right)\right]-E_0[F(X_\bullet)]\Bigg|\nonumber\\
\label{eq:L3_error}
&\leq\lim_{\de\searrow 0}\limsup_{t\to\infty}Q_t^{\nu,h}\left[\De_\de^F\left(\frac{X_{\bullet t}-(2S_{\bullet t}+ht\uk_\bullet)}{\rt{t}}\right)\right].
\end{align}
Unlike \eqref{eq:error_bound} however, we can't use H\"{o}lder's inequality to get a useful bound for \eqref{eq:L3_error} since $(\nu,h)$ is on the critical line $L_3$. Instead, we make use of the Lipschitz continuity of $F$. Suppose $F$ has Lipschitz constant $K$. Then for any $X_\bullet\in \mathcal{C}\big([0,1];\R\big)$ we have
\begin{equation*}
\begin{split}
\De_\de^F(X_\bullet )=\left|F(\Pi_\de X_\bullet)-F(X_\bullet)\right|&\leq K\|\Pi_\de X_\bullet-X_\bullet\|\\
&\leq 2K\sup_{0\leq s\leq\de}|X_s|.
\end{split}
\end{equation*}
Along with \eqref{eq:L3_Bessel}, this implies that \eqref{eq:L3_error} is bounded above by
\begin{equation}\label{eq:Delta_error}
2K\lim_{\de\searrow 0}\limsup_{t\to\infty}\frac{E_0\left[\displaystyle\sup_{0\leq s\leq\de}\left|R_s+h\rt{t}s\right|\exp\left(-h\rt{t}R_1\right)\right]}{E_0\left[\exp\left(-h\rt{t}R_1\right)\right]}.
\end{equation}
With $p_1(\cdot,\cdot)$ denoting the Bessel(3) transition density at time $1$, we can write the expectation appearing in the numerator of \eqref{eq:Delta_error} as a mixture of Bessel(3) bridges by conditioning on the endpoint
\begin{equation}\label{eq:delta_bridge}
\int_0^\infty E_0\left[\sup_{0\leq s\leq\de}\left|R_s+h\rt{t}s\right|\middle|R_1=y\right]\exp(-h\rt{t}y)p_1(0,y)\dd{y}.
\end{equation}
Let $\|\cdot\|_2$ denote the Euclidean norm on $\R^3$. Using \eqref{eq:bes_bridge} and recalling that $h<0$, we can write the above expectation as
$$E\left[\sup_{0\leq s\leq\de}\Bigg|\left\|\left(\bk_s^{(1)}+ys,\bk_s^{(2)},\bk_s^{(3)}\right)\right\|_2-\left\|\left(|h|\rt{t}s,0,0\right)\right\|_2\Bigg|\right]$$
where $\bk^{(i)}$, $i=1,2,3$ are independent Brownian bridges of length $1$ from $0$ to $0$. Now notice that the reverse triangle inequality implies this is bounded above by
$$E\left[\sup_{0\leq s\leq\de}\left\|\left(\bk_s^{(1)}+\big(y+h\rt{t}\big)s,\bk_s^{(2)},\bk_s^{(3)}\right)\right\|_2\right].$$
Using subadditivity of the square root function and the triangle inequality, this is bounded above by
\begin{equation}\label{eq:delta_bound}
3E\left[\sup_{0\leq s\leq\de}\left|\bk_s\right|\right]+\left|y+h\rt{t}\right|\de.
\end{equation}
By substituting \eqref{eq:delta_bound} for the expectation appearing in \eqref{eq:delta_bridge}, we see that the latter expression is bounded above by
\begin{equation}\label{eq:bridge_bound}
3E\left[\sup_{0\leq s\leq\de}\left|\bk_s\right|\right]E_0\left[\exp\left(-h\rt{t}R_1\right)\right]+\de\int_0^\infty\left|y+h\rt{t}\right|\exp(-h\rt{t}y)p_1(0,y)\dd{y}.
\end{equation}
Now substituting \eqref{eq:bridge_bound} for the expectation appearing in the numerator of \eqref{eq:Delta_error} leads to the upper bound
\begin{equation}\label{eq:Delta_split}
2K\lim_{\de\searrow 0}\left(3E\left[\sup_{0\leq s\leq\de}\left|\bk_s\right|\right]+\de\limsup_{t\to\infty}\frac{\int_0^\infty\left|y+h\rt{t}\right|\exp(-h\rt{t}y)p_1(0,y)\dd{y}}{E_0\left[\exp\left(-h\rt{t}R_1\right)\right]}\right).
\end{equation}
We can evaluate the $\limsup$ term appearing in \eqref{eq:Delta_split} explicitly by using the $L_3$ partition function asymptotic from \autoref{prop:asymptotics} and the Bessel(3) transition density formula \eqref{eq:Bes0_dens} to write
\begin{equation*}
\begin{split}
\lim_{t\to\infty}&\frac{\int_0^\infty\left|y+h\rt{t}\right|\exp(-h\rt{t}y)p_1(0,y)\dd{y}}{E_0\left[\exp\left(-h\rt{t}R_1\right)\right]}\\
&=\lim_{t\to\infty}\int_0^\infty\rt{\frac{2}{\pi}}\frac{\left|y+h\rt{t}\right|y^2}{2h^2 t}\exp\left(-h\rt{t}y-\frac{y^2}{2}-\frac{1}{2}h^2 t\right)\dd{y}.
\end{split}
\end{equation*}
Applying the change of variables $y\mapsto y-h\rt{t}$ and using dominated convergence results in
$$\lim_{t\to\infty}\int_{h\rt{t}}^\infty\rt{\frac{2}{\pi}}\frac{|y|\left(y-h\rt{t}\right)^2}{2h^2 t}\exp\left(-\frac{y^2}{2}\right)\dd{y}=\int_{-\infty}^\infty\frac{1}{\rt{2\pi}}|y|\exp\left(-\frac{y^2}{2}\right)\dd{y}=\rt{\frac{2}{\pi}}.$$
Hence \eqref{eq:Delta_split} equals 
$$2K\lim_{\de\searrow 0}\left(3E\left[\sup_{0\leq s\leq\de}\left|\bk_s\right|\right]+\de\rt{\frac{2}{\pi}}\right)=0.$$
Here we used dominated convergence and the fact that $\displaystyle\sup_{0\leq s\leq 1}\left|\bk_s\right|$ is integrable and $\bk$ is continuous with $\bk_0=0$. Now it follows that 
$$\limsup_{t\to\infty}\Bigg|Q_t^{\nu,h}\left[F\left(\frac{X_{\bullet t}-(2S_{\bullet t}+ht\uk_\bullet)}{\rt{t}}\right)\right]-E_0[F(X_\bullet)]\Bigg|=0$$
which proves \eqref{eq:L3_FCLT}.
\end{enumerate}
\end{proof}

\section{Brownian ascent}\label{sec:ascent}

We informally defined the Brownian ascent as a Brownian path of duration $1$ conditioned on the event $\{X_1=S_1\}$. Since this is a null event, some care is needed to make the conditioning precise. Accordingly, we condition on the event $\{S_1-X_1<\ep\}$ and let $\ep\searrow 0$. This leads to an equality in law between the Brownian ascent and a path transformation of the Brownian meander. While this result along with the other Propositions in this section are likely obvious to those familiar with Brownian path fragments, we include the proofs for the convenience of non-experts. The reader can refer to \autoref{app:appendix} for some basic information on the Brownian meander. 
\begin{prop}\label{prop:ascent}
$$(\ak_s:0\leq s\leq 1)\stackrel{\mathcal{L}}{=}(\mk_1-\mk_{1-s}:0\leq s\leq 1)$$
\end{prop}

\begin{proof}
The idea behind the proof is to use the invariance property of $\phi$ from \autoref{def:reversal} together with a known limit theorem for the meander, similarly to proving \eqref{eq:L_2}. Let $F:\mathcal{C}\big([0,1];\R\big)\to\R$ be bounded and continuous. Then we have
\begin{equation*}
\begin{split}
E[F(\ak_\bullet)]&=\lim_{\ep\searrow 0}E_0[F(X_\bullet)|S_1-X_1<\ep]=\lim_{\ep\searrow 0}E_0[F(-X_\bullet)|X_1-I_1<\ep]\\
&=\lim_{\ep\searrow 0}E_0[F_\phi(-X_\bullet)|I_1>-\ep]=E[F_\phi(-\mk_\bullet)]\\
&=E[F(\mk_1-\mk_{1-\bullet})]
\end{split}
\end{equation*}
where weak convergence to Brownian meander in the last limit follows from Theorem 2.1 in \cite{meander_limit}.
\end{proof}

Recall L\'{e}vy's equivalence 
\begin{equation}\label{eq:Levy}
\Big(\big(S_t-X_t,S_t\big):t\geq 0\Big)\stackrel{\mathcal{L}}{=}\Big(\big(|X_t|,L_t^0(X)\big):t\geq 0\Big)
\end{equation}
where $L_t^0(X)$ denotes the local time of $X$ at the level $0$ up to time $t$. This equality in law holds under $P_0$, see Item B in Chapter 1 of \cite{Penalising}. Since conditioning $(X_s:0\leq s\leq 1)$ on the event $\{|X_1|<\ep\}$ and letting $\ep\searrow 0$ results in a standard Brownian bridge, we can apply the same argument of \autoref{prop:ascent} to both sides of \eqref{eq:Levy} and get the following result which can also be seen to follow from a combination of \autoref{prop:ascent} and Th\'{e}or\`{e}me 8 of \cite{meandre}.

\begin{prop}\label{prop:ascent_bridge}
Let $\left(L_s^0(\bk):0\leq s\leq 1\right)$ denote the local time process at the level $0$ of a Brownian bridge of length $1$ from $0$ to $0$. Then we have the equality in law
$$\left(\sup_{0\leq u\leq s}\ak_u :0\leq s\leq 1\right)\stackrel{\mathcal{L}}{=}\left(L_s^0(\bk):0\leq s\leq 1\right).$$
\end{prop}

We can also construct the Brownian ascent from a Brownian path by scaling the pre-maximum part of the path so that it has duration $1$. 

\begin{prop}\label{prop:Den_ascent}
Let $\Te$ denote the almost surely unique time at which the standard Brownian motion $W$ attains its maximum over the time interval $[0,1]$. Then we have the equality in law
$$(\ak_s:0\leq s\leq 1)\stackrel{\mathcal{L}}{=}\left(\frac{W_{s\Te}}{\rt{\Te}}:0\leq s\leq 1\right).$$
\end{prop}

\begin{proof}
From Denisov's path decomposition \autoref{thm:Denisov} we have
$$(\mk_s:0\leq s\leq 1)\stackrel{\mathcal{L}}{=}\left(\frac{W_\Te-W_{\Te-s\Te}}{\rt{\Te}}:0\leq s\leq 1\right).$$
Reflecting both processes about $0$ gives 
$$(-\mk_s:0\leq s\leq 1)\stackrel{\mathcal{L}}{=}\left(\frac{W_{\Te-s\Te}-W_\Te}{\rt{\Te}}:0\leq s\leq 1\right).$$
Applying $\phi$ to both processes results in 
$$(\mk_1-\mk_{1-s}:0\leq s\leq 1)\stackrel{\mathcal{L}}{=}\left(\frac{W_{s\Te}}{\rt{\Te}}:0\leq s\leq 1\right).$$
Now the desired result follows from \autoref{prop:ascent}.
\end{proof}

There is an absolute continuity relation between the path measures of the Brownian ascent and Brownian motion run up to the first hitting time of $1$ and then rescaled to have duration $1$. This random scaling construction is reminiscent of Pitman and Yor's \emph{agreement formula} for Bessel bridges, see \cite{max_decomp}.
\begin{prop}\label{prop:pseudo_ascent}
Let $\tau_1$ be the first hitting time of $1$ by $X$. For any measurable $F:\mathcal{C}\big([0,1];\R\big)\to\R_+$ we have
$$E[F(\ak_\bullet)]=\rt{\frac{\pi}{2}}E_0\left[F\left(\frac{X_{\bullet\tau_1}}{\rt{\tau_1}}\right)\frac{1}{\rt{\tau_1}}\right].$$
\end{prop}

\begin{proof}
\autoref{prop:ascent} and the Imhof relation \eqref{eq:Imhof} imply that
\begin{equation}\label{eq:ascent_Imhof}
E[F(\ak_\bullet)]=\rt{\frac{\pi}{2}}E_0\left[F\left(R_1-R_{1-\bullet}\right)\frac{1}{R_1}\right]
\end{equation}
where $R$ is a Bessel(3) process. Let $\ga_1$ denote the last hitting time of $1$ by $R$. Since $R_\bullet\mapsto F(R_1-R_{1-\bullet})R_1$ is also a non-negative measurable path functional, we can use \autoref{thm:last_hit} to rewrite \eqref{eq:ascent_Imhof} as
\begin{align}
E[F(\ak_\bullet)]=\rt{\frac{\pi}{2}}E_0\left[F\left(R_1-R_{1-\bullet}\right)R_1\frac{1}{R_1^2}\right]&=\rt{\frac{\pi}{2}}E_0\left[F\left(\frac{R_{\ga_1}-R_{(1-\bullet)\ga_1}}{\rt{\ga_1}}\right)\frac{R_{\ga_1}}{\rt{\ga_1}}\right]\nonumber\\
&=\rt{\frac{\pi}{2}}E_0\left[F\left(\frac{1-R_{(1-\bullet)\ga_1}}{\rt{\ga_1}}\right)\frac{1}{\rt{\ga_1}}\right].\label{eq:ascent_rev}
\end{align}
Now Williams' time reversal \autoref{thm:Williams_rev} can be used to conclude that \eqref{eq:ascent_rev} is equal to
$$\rt{\frac{\pi}{2}}E_0\left[F\left(\frac{X_{\bullet\tau_1}}{\rt{\tau_1}}\right)\frac{1}{\rt{\tau_1}}\right].$$
\end{proof}

The process
\begin{equation}\label{eq:co_ascent}
\left(\frac{X_{s\tau_1}}{\rt{\tau_1}}:0\leq s\leq 1\right)
\end{equation}
under $P_0$ which appears in \autoref{prop:pseudo_ascent} has recently been studied by Elie, Rosenbaum, and Yor in \cite{pseudo_ejp, triplet, pseudo_msj, pseudo_esaim}. Among other results, they derive the density of the random variable $\al$ defined by 
$$\al=\frac{X_{U\tau_1}}{\rt{\tau_1}}$$
where $U$ is a Uniform$[0,1]$ random variable independent of $X$. The non-obvious fact that $E_0[\al]=0$ leads to an interesting corollary of \autoref{prop:pseudo_ascent}.

\begin{cor}
Let $U$ be a Uniform$[0,1]$ random variable independent of $\ak$. Then
$$E\left[\int_0^1\frac{\ak_s}{\ak_1}\dd{s}\right]=E\left[\frac{\ak_U}{\ak_1}\right]=0.$$
\end{cor}

\subsection{Brownian co-ascent}

In this section we show how the process \eqref{eq:co_ascent} is related to the Brownian co-meander. The reader unfamiliar with the co-meander can refer to \autoref{app:appendix} for some basic information. The following proposition suggests that a suitable name for the process \eqref{eq:co_ascent} is the \emph{Brownian co-ascent} since it is constructed from the co-meander in the same manner that the ascent is constructed from the meander, viz \autoref{prop:ascent}.

\begin{prop}
If $X$ has distribution $P_0$ then
$$\left(\frac{X_{s\tau_1}}{\rt{\tau_1}}:0\leq s\leq 1\right)
\stackrel{\mathcal{L}}{=}(\tilde{\mk}_1-\tilde{\mk}_{1-s}:0\leq s\leq 1).$$
\end{prop}

\begin{proof}
Let $F:\mathcal{C}\big([0,1];\R\big)\to\R$ be bounded and continuous. Then by Theorem 2.1. in \cite{pseudo_msj} we have
$$E_0\left[F\left(\frac{X_{\bullet\tau_1}}{\rt{\tau_1}}\right)\right]=E_0\left[F(R_1-R_{1-\bullet})\frac{1}{R_1^2}\right]$$
and by \eqref{eq:co_Imhof} we have
$$E_0\left[F(R_1-R_{1-\bullet})\frac{1}{R_1^2}\right]=E[F\left(\tilde{\mk}_1-\tilde{\mk}_{1-\bullet}\right)].$$
The proposition follows from combining these two identities.
\end{proof}

From now on we refer to the process \eqref{eq:co_ascent} as the Brownian co-ascent and denote it by $\left(\tilde{\ak}_s:0\leq s\leq 1\right)$. This allows us to state as an immediate corollary of \autoref{prop:pseudo_ascent} the following absolute continuity relation between the ascent and co-ascent which can also be seen as a counterpart of \eqref{eq:meander_co}.

\begin{cor}
For any measurable $F:\mathcal{C}\big([0,1];\R\big)\to\R_+$ we have
$$E[F(\ak_\bullet)]=\rt{\frac{\pi}{2}}E\left[F(\tilde{\ak}_\bullet)\tilde{\ak}_1\right].$$
\end{cor}

Next we give an analogue of \autoref{prop:ascent_bridge} for the Brownian co-ascent. Let $(\ell_t:t\geq 0)$ denote the inverse local time of $X$ at the level $0$, that is, $\ell_t=\inf\{s:L_s^0(X)>t\}$. The \emph{pseudo-Brownian bridge} $\tilde{\bk}$ was introduced in \cite{pseudo_bridge} and has representation
$$\left(\tilde{\bk}_s:0\leq s\leq 1\right)\stackrel{\mathcal{L}}{=}\left(\frac{X_{s\ell_1}}{\rt{\ell_1}}:0\leq s\leq 1\right)$$
where the right-hand side is under $P_0$, see also \cite{triplet}.

\begin{prop}
Let $\left(L_s^0\big(\tilde{\bk}\big):0\leq s\leq 1\right)$ denote the local time process at the level $0$ of a pseudo-Brownian bridge. Then we have the equality in law
$$\left(\sup_{0\leq u\leq s}\tilde{\ak}_u :0\leq s\leq 1\right)\stackrel{\mathcal{L}}{=}\left(L_s^0\big(\tilde{\bk}\big):0\leq s\leq 1\right).$$
\end{prop}

\begin{proof}
Define $T_1=\inf\{t:S_t=1\}$. Notice that $T_1=\tau_1$ almost surely under $P_0$. Hence
\begin{equation}\label{eq:max_time}
\left(\sup_{0\leq u\leq s}\frac{X_{u\tau_1}}{\rt{\tau_1}}:0\leq s\leq 1\right)\stackrel{\mathcal{L}}{=}\left(\sup_{0\leq u\leq s}\frac{X_{u T_1}}{\rt{T_1}}:0\leq s\leq 1\right)
\end{equation}
under $P_0$. Additionally, L\'{e}vy's equivalence \eqref{eq:Levy} implies
\begin{equation}\label{eq:pseudo_bridge}
\left(\sup_{0\leq u\leq s}\frac{X_{u T_1}}{\rt{T_1}}:0\leq s\leq 1\right)\stackrel{\mathcal{L}}{=}\left(L_s^0\left(\frac{X_{\bullet \ell_1}}{\rt{\ell_1}}\right):0\leq s\leq 1\right)
\end{equation}
under $P_0$. The proposition follows from combining \eqref{eq:max_time} and \eqref{eq:pseudo_bridge}.
\end{proof}

\section{Concluding remarks}\label{sec:future}
Two natural directions for generalizing the main results of this paper are to change the weight process or the reference measure. Scaled penalization of Brownian motion with drift $h\in\R$ by the weight process ${\Ga_t=\exp\big(-\nu(S_t-I_t)\big)}$ is one such possibility. This \emph{range penalization} with $\nu>0$ has been investigated in \cite{Schmock} for $h=0$, in \cite{Povel} for $0<|h|<\nu$, and recently in \cite{Kolb} for $|h|=\nu$. While the first two papers identify the corresponding scaling limit, only partial results are known in the critical case $|h|=\nu$. A related model replaces the Brownian motion with drift by reflecting Brownian motion with drift and penalizes the supremum instead of the range. The asymmetry imposed by the reflecting barrier at $0$ now makes the sign of $h$ relevant. Work in preparation by the current author describes the scaling limit in the critical case for both of these models.

Another interesting question is to what extent can the absolute continuity relation \autoref{prop:pseudo_ascent} and the path constructions \autoref{prop:ascent} and \autoref{prop:Den_ascent} be generalized to processes other than Brownian motion? While all three of these can be nominally applied to many processes, it's not obvious if they yield a bona fide ascent, that is, the process conditioned to end at its maximum. Scale invariance is an underlying theme in all of these results so it makes sense to first consider self-similar processes such as strictly stable L\'{e}vy processes and Bessel processes. In this direction, existing work on stable meanders and the stable analogue of Denisov's decomposition found in Chapter VIII of \cite{Bertoin} would be a good starting point.

\vs\no\textbf{Acknowledgments:} The author would like to thank Iddo Ben-Ari for his helpful suggestions and encouragement and also Jim Pitman and Ju-Yi Yen for their tips on the history of the Brownian meander and co-meander as well as pointers to the literature.

\appendix

\section{Appendix}\label{app:appendix}

Refer to \autoref{sec:notation} for any unfamiliar notation.

\subsection{normalized Brownian excursion, meander and co-meander}

The normalized Brownian excursion, meander and co-meander can be constructed from the excursion of Brownian motion which straddles time $1$. In fact, this is usually how these processes are defined, see Chapter 7 in \cite{local_times}. Define ${g_1=\sup\{t<1:X_t=0\}}$ as the last zero before time $1$ and $d_1=\inf\{t>1:X_t=0\}$ as the first zero after time $1$. Then the normalized excursion, meander and co-meander have representation
$$(\ek_s:0\leq s\leq 1)\stackrel{\mathcal{L}}{=}\left(\frac{|X_{g_1+s(d_1-g_1)}|}{\rt{d_1-g_1}}:0\leq s\leq 1\right),$$
$$(\mk_s:0\leq s\leq 1)\stackrel{\mathcal{L}}{=}\left(\frac{|X_{g_1+s(1-g_1)}|}{\rt{1-g_1}}:0\leq s\leq 1\right)$$
and
$$\left(\tilde{\mk}_s:0\leq s\leq 1\right)\stackrel{\mathcal{L}}{=}\left(\frac{|X_{d_1+s(1-d_1)}|}{\rt{d_1-1}}:0\leq s\leq 1\right),$$
respectively, where the right-hand sides are under $P_0$.

The laws of the meander and co-meander are absolutely continuous with respect to each other and to the law of the Bessel(3) process starting at $0$. More specifically, for any measurable $F:\mathcal{C}\big([0,1];\R\big)\to\R_+$ we have
\begin{align}
E[F(\mk_\bullet)]&=\rt{\frac{\pi}{2}}E_0\left[F(R_\bullet)\frac{1}{R_1}\right]\label{eq:Imhof}\\
E[F(\tilde{\mk}_\bullet)]&=E_0\left[F(R_\bullet)\frac{1}{R_1^2}\right]\label{eq:co_Imhof}\\
E[F(\mk_\bullet)]&=\rt{\frac{\pi}{2}}E\left[F(\tilde{\mk}_\bullet)\tilde{\mk}_1\right].\label{eq:meander_co}
\end{align}
The first of these relations \eqref{eq:Imhof} is known as Imhof's relation \cite{Imhof, Penalising}, while \eqref{eq:co_Imhof} appears as Theorem 7.4.1. in \cite{local_times} and \eqref{eq:meander_co} follows from a combination of the previous two. 

\subsection{Absolute continuity relations}

Here we collect some useful absolute continuity relations between the laws of various processes. While the statements involve bounded measurable path functionals $F$, they are also valid for non-negative measurable $F$. The results given without proof can be found in the literature as indicated. The first two relations give us absolute continuity for Brownian motion and Bessel(3) processes starting at different points, as long as we are willing to ignore the initial $[0,\de]$ segment of the path. See \autoref{def:F_delta} for notation that makes this precise.
\begin{lem}\label{lem:absolute}
Let $x\in\R$, $y>0$, and $0<\de\leq 1$. Then for any bounded measurable $F:\mathcal{C}\big([0,1];\R\big)\to\R$ we have
$$E_0\left[F_\de(X_\bullet)\right]=E_x\left[F_\de(X_\bullet)\exp\left(\frac{x^2-2xX_\de}{2\de}\right)\right]$$
and
$$E_0\left[F_\de(R_\bullet)\right]=E_y\left[F_\de(R_\bullet)\frac{yR_\de\exp\left(\frac{y^2}{2\de}\right)}{\de\sinh\left(\frac{yR_\de}{\de}\right)}\right].$$
\end{lem}
\begin{proof}
We only prove the first statement as the same argument applies to the second; see \eqref{eq:Bes0_dens} and \eqref{eq:Besx_dens} for the Bessel(3) transition densities. First note that by the definition of $F_\de$ and the Markov property we have for any $z\in\R$
$$E_0\left[F_\de(X_\bullet)\middle|X_\de=z\right]=E_x\left[F_\de(X_\bullet)\middle|X_\de=z\right].$$
Now by conditioning on $X_\de$ with $p_\de(\cdot,\cdot)$ denoting the transition density of Brownian motion at time $\de$, we can write
\begin{equation*}
\begin{split}
E_0\left[F_\de(X_\bullet)\right]&=\int_{-\infty}^\infty E_0\left[F_\de(X_\bullet)\middle|X_\de =z\right]p_\de(0,z)\dd{z}\\
&=\int_{-\infty}^\infty E_x\left[F_\de(X_\bullet)\middle|X_\de =z\right]\frac{p_\de(0,z)}{p_\de(x,z)}p_\de(x,z)\dd{z}\\
&=\int_{-\infty}^\infty E_x\left[F_\de(X_\bullet)\frac{p_\de(0,X_\de)}{p_\de(x,X_\de)}\middle|X_\de =z\right]p_\de(x,z)\dd{z}\\
&=E_x\left[F_\de(X_\bullet)\frac{\exp\left(-\frac{X_\de^2}{2\de}\right)}{\exp\left(-\frac{(X_\de-x)^2}{2\de}\right)}\right]\\
&=E_x\left[F_\de(X_\bullet)\exp\left(\frac{x^2-2xX_\de}{2\de}\right)\right].
\end{split}
\end{equation*}
\end{proof}

The next relation results from an $h$-transform of Brownian motion by the harmonic function $h(x)=x$. See Section 1.6 of \cite{local_times}.
\begin{prop}\label{prop:h_abs}
Let $x>0$. Then for any bounded measurable $F:\mathcal{C}\big([0,1];\R\big)\to\R$ we have
$$E_x[F(R_\bullet)]=E_x\left[F(X_\bullet)\frac{X_1}{x};I_1>0\right].$$
\end{prop}

The law of a Bessel(3) process run up to the last hitting time of $x>0$ is, after rescaling, absolutely continuous with respect to the law of a Bessel(3) process run up to a fixed time. This is a special case of Th\'{e}or\`{e}me 3 in \cite{pseudo_bridge}; see also Theorem 8.1.1. in \cite{local_times}.
\begin{thm}\label{thm:last_hit}
Let $\ga_x$ be the last hitting time of $x>0$ by the Bessel(3) process $R$. Then for any bounded measurable $F:\mathcal{C}\big([0,1];\R\big)\to\R$ we have
$$E_0\left[F\left(\frac{R_{\bullet \ga_x}}{\rt{\ga_x}}\right)\right]=E_0\left[F(R_\bullet)\frac{1}{R_1^2}\right].$$
\end{thm}

\subsection{Path decompositions}

Denisov's path decomposition \cite{Denisov} asserts that the pre and post-maximum parts of a Brownian path are rescaled independent Brownian meanders. See Corollary 17 in Chapter VIII of \cite{Bertoin} for an extension to strictly stable L\'{e}vy processes.

\begin{thm}[Denisov]\label{thm:Denisov}\ \\
Let $\Te$ denote the almost surely unique time at which the Brownian motion $W$ attains its maximum over the time interval $[0,1]$. Then the transformed pre-maximum path
$$\left(\frac{W_\Te-W_{\Te-s\Te}}{\rt{\Te}}:0\leq s\leq 1\right)$$
and the transformed post-maximum path
$$\left(\frac{W_\Te-W_{\Te+s(1-\Te)}}{\rt{1-\Te}}:0\leq s\leq 1\right)$$
are independent Brownian meanders which are independent of $\Te$.
\end{thm}

Williams' path decomposition for Brownian motion with drift $h<0$ splits the path at the time of the global maximum $\Te_\infty$ by first picking an Exponential$(-2h)$ distributed $S_\infty$ and then running a Brownian motion with drift $-h$ until it hits the level $S_\infty$ for the pre-maximum path and then running Brownian motion with drift $h$ conditioned remain below $S_\infty$ for the post-maximum path. See Theorem 55.9 in Chapter VI of \cite{R&W_vol2} for the following more precise statement.

\begin{thm}[Williams]\label{thm:Williams}\ \\
Suppose $h<0$ and consider the following independent random elements:
\begin{enumerate}[label={}]
\item $(X_t:t\geq 0)$, a Brownian motion with drift $-h$ starting at $0$; 
\item $(R_t:t\geq 0)$, a Brownian motion with drift $h$ starting at $0$ conditioned to be non-positive for all time;
\item and $l$, an Exponential$(-2h)$ random variable.
\end{enumerate}
Let $\tau_l=\inf\{t:X_t=l\}$ be the first hitting time of the level $l$ by $X$. Then the process 
$$\widetilde{X}_t = \left\{
  \begin{array}{ll}
    X_t & : 0\leq t\leq \tau_l \\
    \\
    l+R_{t-\tau_l} & : \tau_l< t.
  \end{array}
  \right.
$$
is Brownian motion with drift $h$ starting at $0$.
\end{thm} 

Williams' time reversal connects the laws of Brownian motion run until a first hitting time and a Bessel(3) process run until a last hitting time, see Theorem 49.1 in Chapter III of \cite{R&W_vol1}.
\begin{thm}[Williams]\label{thm:Williams_rev}\ \\
Let $\tau_1=\inf\{t:X_t=0\}$ be the first hitting time of $1$ by the Brownian motion $X$ started at $0$. Let $\ga_1=\sup\{t:R_t=1\}$ be the last hitting time of $1$ by the Bessel(3) process $R$ started at $0$. Then the following equality in law holds:
$$(1-X_{\tau_1-t}:0\leq t\leq\tau_1)\stackrel{\mathcal{L}}{=}(R_t:0\leq t\leq\ga_1).$$
\end{thm}

\subsection{Density and distribution formulas}

The Bessel(3) transition density formulas
\begin{equation}\label{eq:Bes0_dens}
p_t(0,y)=\rt{\frac{2}{\pi t^3}}y^2\exp\left(-\frac{y^2}{2t}\right)\dd{y}
\end{equation}
and
\begin{equation}\label{eq:Besx_dens}
p_t(x,y)=\rt{\frac{2}{\pi t}}\frac{y}{x}\sinh\left(\frac{xy}{t}\right)\exp\left(-\frac{x^2+y^2}{2t}\right)\dd{y},
\end{equation}
valid for $y\geq 0$ and $x,t>0$, can be found in Chapter XI of \cite{Rev_Yor}. The density formula for the endpoint of a Brownian meander 
\begin{equation}\label{eq:meander}
P(\mk_1\in\dd{y})=y\exp\left(-\frac{y^2}{2}\right)\dd{y},~y\geq 0
\end{equation}
follows from \eqref{eq:Bes0_dens} together with the Imhof relation \eqref{eq:Imhof}. The well-known arcsine law for the time of the maximum of Brownian motion states that 
\begin{equation}\label{eq:arcsine}
P_0(\Te_1\in\dd{u})=\frac{1}{\pi\rt{u(1-u)}}\dd{u},~0<u<1.
\end{equation}
Using Denisov's path decomposition \autoref{thm:Denisov}, the densities \eqref{eq:meander} and \eqref{eq:arcsine} can be combined to yield the joint density 
\begin{equation}\label{eq:tri_var}
P_0\left(S_1\in\dd{x}, S_1-X_1\in\dd{y}, \Te_1\in\dd{u}\right)=\frac{xy}{\pi\rt{u^3(1-u)^3}}\exp\left(-\frac{x^2}{2u}-\frac{y^2}{2(1-u)}\right)\dd{x}\dd{y}\dd{u}
\end{equation}
which holds for $x,y\geq 0$ and $0<u<1$.

The maximum of a Brownian bridge from $0$ to $a$ of length $T>0$ has distribution
\begin{equation}\label{eq:bridge_max}
P_0\left(\sup_{0\leq s\leq T}X_s\geq b\middle|X_T=a\right)=\exp\left(-\frac{2b(b-a)}{T}\right)
\end{equation}
where $b\geq\max\{0,a\}$, see (4.3.40) in \cite{K&S}.

\subsection{Asymptotic analysis tools}
The following versions of these standard results in asymptotic analysis can be found in \cite{Olver} and \cite{algorithms}, respectively.

\begin{lem}[Watson's lemma]\label{lem:Watson}\ \\
Let $q(x)$ be a function of the positive real variable $x$, such that
$$q(x)\sim\sum_{n=0}^\infty a_n x^\frac{n+\la-\mu}{\mu}~\text{as}~x\to 0,$$
where $\la$ and $\mu$ are positive constants. Then
$$\int_0^\infty q(x)e^{-t x}\dd{x}\sim\sum_{n=0}^\infty \Ga\left(\frac{n+\la}{\mu}\right)\frac{a_n}{t^\frac{n+\la}{\mu}}~\text{as}~t\to\infty.$$
\end{lem}

\begin{thm}[Laplace's method]\label{thm:Laplace}\ \\
Define
$$I(t)=\int_a^b f(x)e^{-t h(x)}\dd{x}$$
where $-\infty\leq a<b\leq\infty$ and $t>0$. Assume that:
\begin{enumerate}[label=\normalfont\roman*.)]
\item $h(x)$ has a unique minimum on $[a,b]$ at point $x=x_0\in(a,b)$,
\item $h(x)$ and $f(x)$ are continuously differentiable in a neighborhood of $x_0$ with $f(x_0)\neq 0$ and
$$h(x)=h(x_0)+\frac{1}{2}h''(x_0)(x-x_0)^2+O\big((x-x_0)^3\big)~\text{as}~x\to x_0,$$
\item the integral $I(t)$ exists for sufficiently large $t$.
\end{enumerate}
Then
$$I(t)=f(x_0)\rt{\frac{2\pi}{t h''(x_0)}}e^{-t h(x_0)}\left(1+O\left(\frac{1}{\rt{t}}\right)\right)~\text{as}~t\to\infty.$$
\end{thm}

\subsection{Convergence lemma}

Here we give a Fatou-type lemma that helps streamline the proofs of the main theorems.

\begin{lem}\label{lem:Fatou}
Suppose $\{F_t\}_{t\geq 0}$, $F$, $\{X_t\}_{t\geq 0}$, and $X$ are all integrable functions defined on the same measure space $(\Omega,\Sigma,\mu)$ such that $\{|F_t|\}_{t\geq 0}$ are bounded by $M>0$, $\{X_t\}_{t\geq 0}$ are non-negative, $\int X_t\dd{\mu}\to \int X\dd{\mu}$, and both $F_t\to F$ and $X_t\to X$ $\mu$-almost surely. Then we have
$$\lim_{t\to\infty}\int F_t X_t\dd{\mu}=\int FX\dd{\mu}.$$
\end{lem}
\begin{proof}
Notice that $(M+F_t)X_t$ is non-negative for all $t\geq 0$. So by Fatou's lemma we have
$$M\int X\dd{\mu}+\liminf_{t\to\infty}\int F_t X_t\dd{\mu}=\liminf_{t\to\infty}\int (M+F_t)X_t\dd{\mu}\geq\int(M+F)X\dd{\mu}$$
which implies
$$\liminf_{t\to\infty}\int F_t X_t\dd{\mu}\geq\int FX\dd{\mu}.$$
Similarly, $(M-F_t)X_t$ is non-negative for all $t\geq 0$, hence  
$$M\int X\dd{\mu}-\limsup_{t\to\infty}\int F_t X_t\dd{\mu}=\liminf_{t\to\infty}\int (M-F_t)X_t\dd{\mu}\geq\int (M-F)X\dd{\mu}$$
which implies
$$\limsup_{t\to\infty}\int F_t X_t\dd{\mu}\leq\int FX\dd{\mu}.$$
Together these inequalities imply that
$$\lim_{t\to\infty}\int F_t X_t\dd{\mu}=\int FX\dd{\mu}.$$
\end{proof}

\bibliography{bibfile}

\providecommand{\bysame}{\leavevmode\hbox to3em{\hrulefill}\thinspace}
\providecommand{\MR}{\relax\ifhmode\unskip\space\fi MR }
\providecommand{\MRhref}[2]{%
  \href{http://www.ams.org/mathscinet-getitem?mr=#1}{#2}
}
\providecommand{\href}[2]{#2}
\begin{thebibliography}{BLGY87}

\bibitem[BCP03]{fp_bridge}
Jean Bertoin, Lo{\"{\i}}c Chaumont, and Jim Pitman, \emph{Path transformations
  of first passage bridges}, Electron. Comm. Probab. \textbf{8} (2003),
  155--166 (electronic). \MR{2042754}

\bibitem[Ber96]{Bertoin}
Jean Bertoin, \emph{L\'evy processes}, Cambridge Tracts in Mathematics, vol.
  121, Cambridge University Press, Cambridge, 1996. \MR{1406564}

\bibitem[BLGY87]{pseudo_bridge}
Ph. Biane, J.-F. Le~Gall, and M.~Yor, \emph{Un processus qui ressemble au pont
  brownien}, S\'eminaire de {P}robabilit\'es, {XXI}, Lecture Notes in Math.,
  vol. 1247, Springer, Berlin, 1987, pp.~270--275. \MR{941990}

\bibitem[BY88]{meandre}
Ph. Biane and M.~Yor, \emph{Quelques pr\'ecisions sur le m\'eandre brownien},
  Bull. Sci. Math. (2) \textbf{112} (1988), no.~1, 101--109. \MR{942801}

\bibitem[CUB11]{bridges3}
Lo{\"{\i}}c Chaumont and Ger{\'o}nimo Uribe~Bravo, \emph{Markovian bridges:
  weak continuity and pathwise constructions}, Ann. Probab. \textbf{39} (2011),
  no.~2, 609--647. \MR{2789508}

\bibitem[Deb09]{Debs1}
Pierre Debs, \emph{Penalisation of the standard random walk by a function of
  the one-sided maximum, of the local time, or of the duration of the
  excursions}, S\'eminaire de probabilit\'es {XLII}, Lecture Notes in Math.,
  vol. 1979, Springer, Berlin, 2009, pp.~331--363. \MR{2599215}

\bibitem[Deb12]{Debs2}
P.~Debs, \emph{Penalisation of the symmetric random walk by several functions
  of the supremum}, Markov Process. Related Fields \textbf{18} (2012), no.~4,
  651--680. \MR{3051657}

\bibitem[Den83]{Denisov}
I.~V. Denisov, \emph{Random walk and the {W}iener process considered from a
  maximum point}, Teor. Veroyatnost. i Primenen. \textbf{28} (1983), no.~4,
  785--788. \MR{726906}

\bibitem[DIM77]{meander_limit}
Richard~T. Durrett, Donald~L. Iglehart, and Douglas~R. Miller, \emph{Weak
  convergence to {B}rownian meander and {B}rownian excursion}, Ann. Probability
  \textbf{5} (1977), no.~1, 117--129. \MR{0436353}

\bibitem[ERY14]{pseudo_ejp}
Romuald Elie, Mathieu Rosenbaum, and Marc Yor, \emph{On the expectation of
  normalized {B}rownian functionals up to first hitting times}, Electron. J.
  Probab. \textbf{19} (2014), no. 37, 23. \MR{3194736}

\bibitem[FPY93]{bridges1}
Pat Fitzsimmons, Jim Pitman, and Marc Yor, \emph{Markovian bridges:
  construction, {P}alm interpretation, and splicing}, Seminar on {S}tochastic
  {P}rocesses, 1992 ({S}eattle, {WA}, 1992), Progr. Probab., vol.~33,
  Birkh\"auser Boston, Boston, MA, 1993, pp.~101--134. \MR{1278079}

\bibitem[Imh84]{Imhof}
J.-P. Imhof, \emph{Density factorizations for {B}rownian motion, meander and
  the three-dimensional {B}essel process, and applications}, J. Appl. Probab.
  \textbf{21} (1984), no.~3, 500--510. \MR{752015}

\bibitem[KS88]{K&S}
Ioannis Karatzas and Steven~E. Shreve, \emph{Brownian motion and stochastic
  calculus}, Graduate Texts in Mathematics, vol. 113, Springer-Verlag, New
  York, 1988. \MR{917065}

\bibitem[KS17]{Kolb}
Martin Kolb and Mladen Savov, \emph{Conditional survival distributions of
  {B}rownian trajectories in a one dimensional {P}oissonian environment in the
  critical case}, Electron. J. Probab. \textbf{22} (2017), Paper No. 14, 29.
  \MR{3622884}

\bibitem[NRY09]{global}
J.~Najnudel, B.~Roynette, and M.~Yor, \emph{A global view of {B}rownian
  penalisations}, MSJ Memoirs, vol.~19, Mathematical Society of Japan, Tokyo,
  2009. \MR{2528440}

\bibitem[Olv74]{Olver}
F.~W.~J. Olver, \emph{Asymptotics and special functions}, Academic Press [A
  subsidiary of Harcourt Brace Jovanovich, Publishers], New York-London, 1974,
  Computer Science and Applied Mathematics. \MR{0435697}

\bibitem[OY06]{beyond_harmonic}
Jan Ob\l\'oj and Marc Yor, \emph{On local martingale and its supremum: harmonic
  functions and beyond}, From stochastic calculus to mathematical finance,
  Springer, Berlin, 2006, pp.~517--533. \MR{2234288}

\bibitem[Pit06]{Pit_lec}
J.~Pitman, \emph{Combinatorial stochastic processes}, Lecture Notes in
  Mathematics, vol. 1875, Springer-Verlag, Berlin, 2006, Lectures from the 32nd
  Summer School on Probability Theory held in Saint-Flour, July 7--24, 2002,
  With a foreword by Jean Picard. \MR{2245368}

\bibitem[Pov95]{Povel}
Tobias Povel, \emph{On weak convergence of conditional survival measure of
  one-dimensional {B}rownian motion with a drift}, Ann. Appl. Probab.
  \textbf{5} (1995), no.~1, 222--238. \MR{1325050}

\bibitem[Pro15]{Profeta}
Christophe Profeta, \emph{Some limiting laws associated with the integrated
  {B}rownian motion}, ESAIM Probab. Stat. \textbf{19} (2015), 148--171.
  \MR{3386368}

\bibitem[PY96]{max_decomp}
Jim Pitman and Marc Yor, \emph{Decomposition at the maximum for excursions and
  bridges of one-dimensional diffusions}, It\^o's stochastic calculus and
  probability theory, Springer, Tokyo, 1996, pp.~293--310. \MR{1439532}

\bibitem[RVY05]{long_bridges}
Bernard Roynette, Pierre Vallois, and Marc Yor, \emph{Limiting laws for long
  {B}rownian bridges perturbed by their one-sided maximum. {III}}, Period.
  Math. Hungar. \textbf{50} (2005), no.~1-2, 247--280. \MR{2162812}

\bibitem[RVY06]{supremum}
\bysame, \emph{Limiting laws associated with {B}rownian motion perturbed by its
  maximum, minimum and local time. {II}}, Studia Sci. Math. Hungar. \textbf{43}
  (2006), no.~3, 295--360. \MR{2253307}

\bibitem[RW00a]{R&W_vol1}
L.~C.~G. Rogers and David Williams, \emph{Diffusions, {M}arkov processes, and
  martingales. {V}ol. 1}, Cambridge Mathematical Library, Cambridge University
  Press, Cambridge, 2000, Foundations, Reprint of the second (1994) edition.
  \MR{1796539}

\bibitem[RW00b]{R&W_vol2}
\bysame, \emph{Diffusions, {M}arkov processes, and martingales. {V}ol. 2},
  Cambridge Mathematical Library, Cambridge University Press, Cambridge, 2000,
  It{\^o} calculus, Reprint of the second (1994) edition. \MR{1780932}

\bibitem[RY94]{Rev_Yor}
Daniel Revuz and Marc Yor, \emph{Continuous martingales and {B}rownian motion},
  second ed., Grundlehren der Mathematischen Wissenschaften [Fundamental
  Principles of Mathematical Sciences], vol. 293, Springer-Verlag, Berlin,
  1994. \MR{1303781}

\bibitem[RY09]{Penalising}
Bernard Roynette and Marc Yor, \emph{Penalising {B}rownian paths}, Lecture
  Notes in Mathematics, vol. 1969, Springer-Verlag, Berlin, 2009. \MR{2504013}

\bibitem[RY14]{triplet}
Mathieu Rosenbaum and Marc Yor, \emph{On the law of a triplet associated with
  the pseudo-{B}rownian bridge}, S\'eminaire de {P}robabilit\'es {XLVI},
  Lecture Notes in Math., vol. 2123, Springer, Cham, 2014, pp.~359--375.
  \MR{3330825}

\bibitem[RY15a]{pseudo_msj}
\bysame, \emph{Random scaling and sampling of {B}rownian motion}, J. Math. Soc.
  Japan \textbf{67} (2015), no.~4, 1771--1784. \MR{3417513}

\bibitem[RY15b]{pseudo_esaim}
\bysame, \emph{Some explicit formulas for the {B}rownian bridge, {B}rownian
  meander and {B}essel process under uniform sampling}, ESAIM Probab. Stat.
  \textbf{19} (2015), 578--589. \MR{3433427}

\bibitem[Sch90]{Schmock}
Uwe Schmock, \emph{Convergence of the normalized one-dimensional {W}iener
  sausage path measures to a mixture of {B}rownian taboo processes},
  Stochastics Stochastics Rep. \textbf{29} (1990), no.~2, 171--183.
  \MR{1041034}

\bibitem[Szp01]{algorithms}
Wojciech Szpankowski, \emph{Average case analysis of algorithms on sequences},
  Wiley-Interscience Series in Discrete Mathematics and Optimization,
  Wiley-Interscience, New York, 2001, With a foreword by Philippe Flajolet.
  \MR{1816272}

\bibitem[Yan13]{stable_pen2}
Yuko Yano, \emph{A remarkable {$\sigma$}-finite measure unifying supremum
  penalisations for a stable {L}\'evy process}, Ann. Inst. Henri Poincar\'e
  Probab. Stat. \textbf{49} (2013), no.~4, 1014--1032. \MR{3127911}

\bibitem[YY13]{local_times}
Ju-Yi Yen and Marc Yor, \emph{Local times and excursion theory for {B}rownian
  motion}, Lecture Notes in Mathematics, vol. 2088, Springer, Cham, 2013, A
  tale of Wiener and It\^o measures. \MR{3134857}

\bibitem[YYY10]{stable_pen1}
Kouji Yano, Yuko Yano, and Marc Yor, \emph{Penalisation of a stable {L}\'evy
  process involving its one-sided supremum}, Ann. Inst. Henri Poincar\'e
  Probab. Stat. \textbf{46} (2010), no.~4, 1042--1054. \MR{2744885}

\end{thebibliography}
\bibliographystyle{amsalpha}

\end{document}